\documentclass[12pt,onecolumn,draftcls,journal,letterpaper]{IEEEtran}
\IEEEoverridecommandlockouts
\usepackage{amssymb,amsmath,color}
\usepackage{graphicx}
\usepackage{epsfig}
\usepackage{mathdots}
\usepackage{subfig}

\renewcommand{\natural}{{\mathbb{N}}}

\newcommand{\real}{{\mathbb{R}}}

\newcommand{\union}{\cup}
\newcommand{\intersection}{\ensuremath{\operatorname{\cap}}}

\newcommand{\map}[3]{#1: #2 \rightarrow #3}

\newcommand{\setdef}[2]{\{#1 \; | \; #2\}}

\newcommand{\CC}{\mathcal{C}}

\newcommand{\NN}{\mathcal{N}}

\newcommand{\until}[1]{\{1,\dots,#1\}}
\newcommand{\fromto}[2]{\{#1,\dots,#2\}}

\newcommand{\degree}{\textup{\text{d}}}

\newcommand{\Lu}{N}

\newcommand{\eqmod}[1]{{\buildrel\rm mod \, #1\over=}}

\newtheorem{theorem}{Theorem}[section]
\newtheorem{proposition}[theorem]{Proposition}
\newtheorem{corollary}[theorem]{Corollary}
\newtheorem{definition}[theorem]{Definition}
\newtheorem{lemma}[theorem]{Lemma}
\newtheorem{remark}[theorem]{Remark}

\newcommand\oprocendsymbol{\hbox{$\square$}}
\newcommand\oprocend{\relax\ifmmode\else\unskip\hfill\fi\oprocendsymbol}

\graphicspath{{figs/}}

\begin{document}

% \begin{frontmatter}

\title{Controllability and observability of grid graphs via reduction and
  symmetries \thanks{Early short versions of this work appeared
    as~\cite{GN-GP:11a} and \cite{GN-GP:11b}: differences between this early
    short version and the current article include a much improved comprehensive
    and thorough treatment, revised complete proofs for all statements.}  
}

\author{Giuseppe Notarstefano \and Gianfranco Parlangeli \thanks{The research
    leading to these results has received funding from the European Community's
    Seventh Framework Programme (FP7/2007-2013) under grant agreement no. 224428
    (CHAT) and n. 231378 (CO3AUV).} 
  % and from the Italian Minister under the
  % national project Sviluppo di nuovi metodi e algoritmi per l'identificazione,
  % la stima bayesiana e il controllo adattativo e distribuito.}
  \thanks{Gianfranco Parlangeli and Giuseppe Notarstefano are with the
    Department of Engineering, University of Lecce, Via per Monteroni, 73100
    Lecce, Italy, \texttt{\{gianfranco.parlangeli,
      giuseppe.notarstefano\}@unile.it}}}

% \title{Preliminaries on the observability and controllability of grid and torus
% graphs \thanksref{footnoteinfo}} % Title, preferably not more

% \thanks[footnoteinfo]{The research leading to these results has received
% funding from the European Community's Seventh Framework Programme
% (FP7/2007-2013) under grant agreement no. 224428 (CHAT) and n. 231378 (CO3AUV)
% and from the Italian Minister under the national project Sviluppo di nuovi
% metodi e algoritmi per l'identificazione, la stima bayesiana e il controllo
% adattativo e distribuito.}

% \author[First]{Giuseppe Notarstefano} \author[First]{Gianfranco Parlangeli}

% \address[First]{Department of Engineering, University of Lecce,\\ Via per
% Monteroni, 73100 Lecce, Italy,\\ \texttt{\{gianfranco.parlangeli,
% giuseppe.notarstefano\}@unile.it}} \address[Second]{Colorado State University,
% Fort Collins, CO 80523 USA (e-mail: author@lamar. colostate.edu)}

% \begin{keyword} % Five to ten keywords,
%   Laplacian, consensus, controllability, observability % chosen from the IFAC
% \end{keyword} % keyword list or with the
% help of the Automatica
% keyword wizard

\maketitle

\begin{abstract}
  In this paper we investigate the controllability and observability properties
  of a family of linear dynamical systems, whose structure is induced by the
  Laplacian of a grid graph. This analysis is motivated by several applications
  in network control and estimation, quantum computation and discretization of
  partial differential equations.  Specifically, we characterize the structure
  of the grid eigenvectors by means of suitable decompositions of the graph. For
  each eigenvalue, based on its multiplicity and on suitable symmetries of the
  corresponding eigenvectors, we provide necessary and sufficient conditions to
  characterize all and only the nodes from which the induced dynamical system is
  controllable (observable).
  We discuss the proposed criteria and show, through suitable examples, how such
  criteria reduce the complexity of the controllability (respectively
  observability) analysis of the grid.
\end{abstract}

\section{Introduction}

% MOTIVATION
In several modern engineering application areas, there are important physical
phenomena whose dynamic model is induced or strictly related to the structure of
a graph that models the interaction among components of the main system.

For example, in multi-agent system control (e.g., distributed robotics, sensor
networks or smart power grids) a communication or interaction graph induces the
structure of the feedback control input. In Markov chains the evolution of the
probabilities of a finite number of states can be modeled as a dynamical system
structured according to the graph of one-step transition probabilities.
Similarly, in quantum computation the evolution of interacting particles obeying
to quantum laws may be described by differential equations structured according
to an interaction graph. Another area in which a dynamic model is related
to a graph structure is the one of discretized partial differential
equations. In this case the graph determined by the discretization rule induces
the structure of the approximating ordinary differential equations.

In this paper we will concentrate on linear time invariant dynamical systems
whose state matrix is induced by the Laplacian of a fixed
undirected graph. In particular, we consider $d$-dimensional grid graphs, also
know as \emph{lattices}. This graph topology appears in many important
application scenarios as we show in the next sections.

We investigate the eigenstructure of the grid graph Laplacian in terms of
structural properties of the graph, and study how this structure affects the
controllability and observability properties of the induced system. Our main goal
is to relate the controllability and observability properties to graph theoretic
rules involving simple arithmetic operations on the graph labeling.

% LITERATURE REVIEW
Controllability of complex network systems has received a
widespread attention in the last years in several areas
\cite{YYL-JJS-ALB:11}. We organize the relevant literature in three parts
according to the three main motivating scenarios for our problem set-up. First,
a system with Laplacian state dynamics arises in network systems running an
\emph{average consensus} algorithm.
A survey on these algorithms and their performance may be found in
\cite{ROS-JAF-RMM:07} and references therein.
The controllability problem for a leader-follower network was
introduced in \cite{HGT:04} for a single control node.
Intensive simulations were provided showing that it is ``unlikely'' for a
Laplacian based consensus network to be completely controllable.
In \cite{GP-GN:12}, see also \cite{GP-GN:10b}, ``necessary and sufficient''
conditions are provided to characterize the controllability and observability of
path and cycle graphs in terms of simple rules from number theory.
In \cite{AR-MJ-MM-ME:09} and \cite{SM-ME-AB:09}, see also \cite{MM-ME:10},
necessary conditions for controllability, based on suitable properties of the
graph, are provided. 
% The conditions rely on algebraic graph tools based on the notion of equitable
% partitions of a graph.  A more exhaustive analysis of the controllability and
% other structural system properties for network systems on the basis of graph
% structural properties can be found in \cite{MM-ME:10}.
%
Other contributions on the controllability of network systems can be found in
\cite{BL-TC-LW-GX:08,ZJ-ZW-HL-ZW:10,ZJ-HL-THL-QL:10}.
Observability has been studied for the first time in \cite{MJ-ME:07}, where
necessary conditions for observability, as in the dual controllability setting
investigated in \cite{AR-MJ-MM-ME:09} and \cite{SM-ME-AB:09}, are provided.
%
% A recent reference on observability for network dynamic systems is
% \cite{DZ-MM:08}. Here, the linear dynamical systems of the network are
% decoupled and the coupling among the systems appears through the
% output. % Thus, the observability issues turn to be different from the one we
% investigate (in our network the systems dynamics are coupled and the output
% decoupled).
%
A parallel research line investigates slightly different properties called
\emph{structural controllability}, \cite{MZ-HL:09,SJ-AA-AGA:11}, and
\emph{structural observability}, \cite{SS-CH:08}. Here, the objective is to
choose the nonzero entries of the consensus matrix (i.e. the state matrix of the
resulting network system) in order to obtain observability from a given set of
nodes. 
% However, in many contexts the structure of the system matrix is given
% (e.g. the Laplacian for average consensus). % Thus, we believe that the
% problem studied in the paper is of interest.
%
It is worth noting that controllability and observability of a network system are
necessary structural properties in many interesting network problem as estimation,
intrusion detection and formation control, e.g., \cite{SS-CH:08,
  FP-AB-FB:10, FP-RC-AB-FB:10b, mj-am-me:06}.

Second, continuous time quantum walks can be modeled as linear time invariant
systems whose state matrix is the imaginary skew-Hermitian matrix $i H$, where
$H$ is called Hamiltonian and can be either the Adjacency matrix or the
Laplacian of the transition graph \cite{JK:03}. %, \cite{HG-JW:03}.

% Structural properties of the Laplacian of Cayley graphs where studied in\cite{HG-JW:03} to
% exploit differences between quantum and classical random walks on regular graphs
% and how they.

The state transfer problem (which is strictly related to the controllability
problem) for quantum systems is investigated in \cite{CG:12}. The paper
explores the eigenstructure of the Hamiltonian (which is taken as the Adjacency
matrix of the underline graph) to characterize the state transfer.  A key
reference establishing a connection between our controllability analysis and the
controllability of quantum walks is \cite{DB-DD-LH-SS-MY:11}. Here the
controllability of continuous time quantum walks is investigated and related to
the controllability of a linear time invariant system with the structure
considered in our paper. The controllability of continuous time quantum walks on
graphs is also studied in \cite{CG-SS:10} and \cite{FA-DD:10}.
More specifically, $d$-dimensional grid or lattice graphs play an important role in quantum
computation. The controllability problem on this specific graph structure has
been investigated in \cite{FA-DD:09} and \cite{CG:12}.

Third and final a system with the structure studied in the paper appears when
discretizing partial differential equations (PDEs) containing the Laplace operator
\cite{ANT-AAS:63}. Such systems include several diffusion and wave propagation
equations appearing in fluid-dynamics, mechanics, acoustics and
electromagnetism. In \cite{RM:98} discretization of PDEs was indicated as a motivating
example for the analysis of the Laplacian eigenstructure. The controllability of
a discretized version of the heat equation on a one dimensional grid domain is
investigated in \cite{SA-MS-BG-KTY:04} and extended to the case of constrained
input in \cite{JSR:07}. Finally, in \cite{TM:11} and \cite{TM-MK:11} trajectory planning of multi-agent
systems is performed by studying a partial differential equation describing a
continuum of agents. That is, the multi-agent dynamics is obtained as a
discretized version of a partial differential equation. Controllability is
guaranteed by the particular choice of the control nodes. A more general choice of the
control nodes leads to our controllability problem.

% CONTRIBUTIONS
The contribution of the paper is threefold. First, we identify a mathematical
framework, namely the controllability and observability of linear time invariant
systems induced by the Laplacian of a grid graph, that has numerous applications
in several engineering areas. In particular, we highlight how this framework
appears in distributed control, quantum computation and discretized partial
differential equations.

Second, we characterize the structure of the Laplacian eigenvectors of a
grid. Namely, we show that, on the basis of a prime number factorization of the
grid dimensions, the eigenvector components present symmetries related to
suitable partitions of the main grid into sub-grids that we call
\emph{bricks}. Given a partition of the grid graph into bricks, we show that the
eigenvalues of the elementary brick are also eigenvalues of the main grid. Also, the grid
eigenvectors associated to the common eigenvalues are obtained by composing (with
suitable flip operations) the corresponding eigenvectors of the basic brick.
Furthermore, we show that in each brick (and thus also in the main grid) the
eigenvector components may show symmetries with respect to one or more of the
grid axes. 
% This property is induced by symmetries of the eigenvectors of path
% graphs that generate the grid graph.

% An important consequence of this characterization is
% that it allows to determine what are the eigenvector components that can be set
% to zero simultaneously. This feature provides a key tool to characterize the
% controllability and observability of the grid graph induced systems under
% investigation. The analysis of the zero eigenvector components plays also an
% important role in the exploration of reduced order dynamic behaviors of this
% class of systems.

Third and final, we provide necessary and sufficient conditions to completely
characterize the controllability and observability of grid graphs. We start showing
that loss of controllability and observability can be studied by identifying all
the zero components of an eigenvector. Based on the evaluation of suitable sets
of polynomials, together with the eigenvector symmetries, we are able to determine
all and only the eigenvector components that can be set to zero
simultaneously. Thus, on the basis of the node labels, the eigenvector
symmetries and the polynomial evaluations, we provide easily implementable
routines to: (i) identify all and only the controllable (observable) nodes of the
graph, (ii) say if the graph is controllable (observable) from a given set of nodes
and (iii) construct a set of control (observation) nodes from which the graph is
controllable (observable).

% {\color{red} Though the starting point of this paper relies on the results in
% \cite{GP-GN:10b}, the methodologies are in fact novel and are based on
% mathematical tools from cartesian product of graphs and Kronecker product of
% matrices.  }

% For space constrains all proofs are omitted in this paper and will be provided
% in a forthcoming document.

  % ORGANIZATION
The paper is organized as follows. In Section~\ref{sec:prelims} we introduce
preliminary definitions and properties of undirected graphs, set up the
controllability and observability problems and describe the motivating scenarios
for our framework.
In Section~\ref{sec:simple_grids} we characterize the controllability and
observability for grid graphs with simple eigenvalues.
  % we investigate suitable symmetries of the path eigenvectors that are at the
  % basis of the new results on grid graphs.
  % 
In Section~\ref{sec:general_grids} we analyze the symmetries in the structure of
the grid graph eigenvectors. On this basis, in
Section~\ref{sec:general_grids_observ} we provide necessary and sufficient
conditions for the controllability (observability) of general grid graphs.
Finally, in Appendix we recall results from \cite{GP-GN:12} on the
controllability (observability) of path graphs.

\paragraph*{Notation}
Let $\natural$ denote the natural numbers, for $i\in\natural$ we let $e_i$ be
the $i$-th element of the canonical basis, e.g. $e_1 = [1~ 0~ \ldots~ 0]^T$. For
a vector $v\in\real^d$ we denote $(v)_i$ the $i$th component of $v$ so that $v =
[(v)_1 \ldots (v)_d]^T$. We denote $\Pi \in \real^{d\times d}$ the
permutation matrix reversing all the components of $v$ so that $\Pi v = [(v)_d
\ldots (v)_1]^T$ (the $j$-th column of $\Pi$ is $[\Pi]_j= e_{n-j+1}$).
Adopting the usual terminology of number theory, we say that $k$ is a
\emph{factor} of $m$ if there is an integer $q$ such that $m=kq$. Given two
integers $b$ and $c$, if an integer $m$ is a factor of $b-c$, we write $b
\eqmod{m} c$.
%
% If two integers $b$ and $c$ satisfy for a given $m$ the relation $m|(b-c)$ then
% we say that $b$ is congruent to $c$ modulo $m$ (written $b = c$ mod($m$) or
% equivalently $b \eqmod{m} c$).  
%
We denote $GCD(a,b)$ the \emph{greatest common divisor} of two positive integers
$a$ and $b$.

%\section{Preliminaries and problem set-up}
\section{Problem set-up and motivations}
\label{sec:prelims}
In this section we present some preliminary terminology on graph theory,
introduce the network model, set up the controllability and observability problems
and provide some standard results for linear systems that will be useful to
prove the main results of the paper.

\subsection{Preliminaries on graph theory}
Let $G = (I, E)$ be a static undirected graph with set of nodes $I=\until{n}$
and set of edges $E\subset I\times I$. We denote $\NN_i$ the set of neighbors of
agent $i$, that is, $\NN_i = \setdef{j\in I}{(i,j)\in E}$, and $\degree_i =
\sum_{j\in \NN_i} 1$ the degree of node $i$.
% The maximum degree of the graph is defined as $\Delta = \max_{i\in I} \degree_i$. 
The degree matrix $D$ of the graph $G$ is the diagonal matrix defined as
$[D]_{ii} = \degree_i$.
The adjacency matrix $A \in \real^{n\times n}$ associated to the graph $G$ is
defined as
\[
\left[ A \right]_{ij} =
\begin{cases}
  1 & \text{if}~ (i,j)\in E\\
  0 & \text{otherwise}.
\end{cases}
\]
The Laplacian $L$ of $G$ is defined as $L = D - A$. The Laplacian is a symmetric
positive semidefinite matrix with $k$ eigenvalues in $0$, where $k$ is the
number of connected components of $G$. If the graph is connected the eigenvector
associated to the eigenvalue $0$ is the vector $\mathbf{1} = [1~ \ldots~ 1]^T$.

Next, we introduce the notion of cartesian product of graphs. Let $G = (I,E)$
and $G'= (I', E')$ be two undirected graphs. The cartesian product $G \, \Box \,
G'$ is a graph with vertex set $I\times I'$ (i.e. the cartesian product of the
two vertex sets) and edge set defined as follows. Nodes $[i,i']\in I\times I'$
and $[k,k']\in I\times I'$ are adjacent in $G \, \Box \, G'$ if either $i=k$ and
$(i',k')\in E'$ or $i'=k'$ and $(i,k)\in E$. The cartesian product is
commutative and associative. Thus, a $d\in\natural$ dimensional product graph,
$\prod_{\ell=1}^{d} G_\ell$, is constructed by combining the above definition
with the associative property.

% An intuitive graphical interpretation of the ...

We introduce the special graphs that will be of interest in the rest of the
paper.
% , namely the path and cycle graphs, and their cartesian products the grid
% graph (path $\Box$ ... $\Box$ path), the torus (cycle $\Box$ ... $\Box$ cycle)
% and cylinder (grid $\Box$ torus).
%
A \emph{path graph} is a graph in which there are only nodes of degree two
except for two nodes of degree one.
The nodes of degree one are called external nodes, while the other are called
internal nodes.
From now on, without loss of generality, we will label the external nodes with
$1$ and $n$, and the internal nodes so that the edge set is $E = \{(i,i+1) \; |
\; i\in\fromto{1}{n-1}\}$.

%
% A \emph{cycle graph} is a graph in which all the nodes have degree two. From
% now on, without loss of generality, we will label the nodes so that the edge
% set is $E = \{(i,i\,\text{mod}(n) + 1) \; | \; i\in\fromto{1}{n}\}$.
%
A $d$-dimensional \emph{grid graph} (or \emph{lattice graph}) is the cartesian
product of $d$ paths (of possibly different length). In a grid graphs the nodes
have degree from $d$ up to $2d$. We call the nodes with degree $d$ \emph{corner
  nodes}. Corner nodes are obtained from the product of external nodes in the
paths.
%
% A $d$-dimensional \emph{torus graph} is the cartesian product of $d$ cycle
% graphs.
%
% Finally, a $(d_1,d_2)$ dimensional \emph{cylinder graph} is the cartesian
% product of $d_1$ path graphs and $d_2$ cycle graphs.

Given a $d$-dimensional grid
graph $G = P_1 \Box \ldots \Box P_d$, we denote $i = [(i)_1, \ldots, (i)_d]$ a node
of $G$, where the component $(i)_\kappa$ identifies the position of the node on
the $\kappa$th path. Also, given a Laplacian eigenvector of the $G$, $w \in
\real^{n_1 \ldots n_d}$, we say ``the component $[(i)_1, \ldots, (i)_d]$ of $w$''
meaning ``the component $(i)_1 \cdot (n_1 \cdot n_2 \cdot \ldots \cdot n_d) +
(i)_2 \cdot (n_2 \cdot \ldots \cdot n_d) \ldots + (i)_d$ of $w$''.

\subsection{Controllability and observability of graph induced systems: problem
  set-up and analysis tools}
Next, we introduce the class of systems that we investigate in the
paper. Informally, we consider linear time invariant systems whose state matrix
is the Laplacian of a grid graph, the input matrix is obtained by directly
controlling a subset of the node dynamics and the output matrix by observing a
subset of the node states. Formally, let $G=(I=\until{n},E)$ be a grid graph, $I_c =
\{i_1, \ldots, i_m\}\subset I$ and $I_o = \{j_1, \ldots, j_p\}\subset I$, a
first order dynamical system induced by $G$, $I_c$ and $I_o$ is the system
\begin{equation}
  \begin{split}
  \dot{x}(t) &= \alpha L x(t) + B u(t),\\
   y(t) &= C x(t),
\end{split}
  \label{eq:induced_system}
\end{equation}
where $\alpha$ is a scalar, $L$ is the Laplacian of $G$, 
% \begin{equation*}
  $B =
  \begin{bmatrix}
    e_{i_1} \; \vline & \ldots & \vline \; e_{i_{m}}
  \end{bmatrix}$ and
$C =
  \begin{bmatrix}
    e_{i_1} \; \vline & \ldots & \vline \; e_{i_{p}}
  \end{bmatrix}^T$.
% \end{equation*}

%
It is a well known result in linear systems theory that the observability
properties of the state-output pair $(L, C)$ correspond to the controllability
properties of the state-input pair $(L^T, C^T) = (L, C^T)$. Thus, the
controllability and observability analysis for the class of systems in
\eqref{eq:induced_system} can be performed by using the same tools.

% \subsection{Standard results on controllability and observability of linear
%   systems}
We start with some notation. The set of states that are controllable is the \emph{controllable
  subspace} and will be denoted $X_c$. Respectively, the set of initial states
that produce an identically zero output is the \emph{unobservable subspace} and
will be denoted $X_{no}$.

An important result on the controllability (observability) of
time-invariant linear systems is the Popov-Belevitch-Hautus (PBH) lemma, e.g. \cite{PJA-ANM:94}.
Combining the PBH lemma with the fact that the state matrix is symmetric
(therefore diagonalizable) the following lemma follows.

\begin{lemma}[PBH lemma for symmetric matrices]
  Let $A\in \real^{n\times n}$, $B \in \real^{n\times m}$ and
  $C\in\real^{p\times n}$, $n,m,p \in \natural$, be the state, input and output
  matrices of a linear time-invariant system, where $A$ is symmetric. Then, the
  unobservable subspace $X_{no}$ associated to the pair $(A, C)$ (respectively
  the orthogonal complement to the controllable subspace $X_c$ associated to the
  pair $(A,B)$) is spanned by vectors $v_l$ satisfying for some
  $\lambda\in\real$
  \begin{equation}
    \begin{matrix}
      B^T v_l &= 0\\
      A v_l &= \lambda v_l,
    \end{matrix}
    \quad\text{respectively}\quad
   \begin{matrix}
      C v_l &= 0\\
      A v_l &= \lambda v_l.
    \end{matrix}
    \label{eq:PBH_eigvec_reach_obs}
  \end{equation}
  That is, the basis vectors of $X_{no}$ ($X_c^\perp$) are the eigenvectors of
  $A$ with zero in the $j_1$-th, $\ldots$, $j_p$-th ($i_1$-th, $\ldots$, $i_m$-th)
  components. \oprocend
  \label{lmm:PBH_eigvec}
\end{lemma}

In the rest of the paper we will denote the eigenvalues and eigenvectors for
which \eqref{eq:PBH_eigvec_reach_obs} holds \emph{uncontrollable} (respectively,
\emph{unobservable}) \emph{eigenvalues and eigenvectors}.

\begin{remark}[Higher order integrators]
  The controllability (observability)
  analysis for $k$-th order integrators of the form $x^{(k)}(t) = \alpha L x(t) + B u(t)$ is
  equivalent to the analysis of the first order system
  \eqref{eq:induced_system}. The statement follows, for example, 
  % by exploiting the particular block structure of the controllability matrix.
  by writing the conditions in Lemma~\ref{lmm:PBH_eigvec} for the $k$-th order integrator system and exploiting its
  block structure. \oprocend
\end{remark}

\subsection{Motivating applications}
Next, we show three main areas of application for our results.

\subsubsection*{\bf Network of agents running average consensus}
We consider a collection of agents labeled by a set of identifiers $I =
\until{n}$, where $n\in\natural$ is the number of agents. We assume that the
agents communicate according to a \emph{time-invariant undirected} communication
graph $G = (I, E)$, where $E = \setdef{(i,j)\in I\times I}{i~ \text{and}~ j~
  \text{communicate}}$. % That is, we assume that the communication between any two agents is bi-directional.
The agents run a consensus algorithm based on a Laplacian control law (see
e.g. \cite{ROS-JAF-RMM:07} for a survey). The dynamics of the agents evolve in
continuous time ($t\in\real_{\geq0}$) and are given by
\begin{equation*}
  \begin{split}
    \dot{x}_i(t) = - \sum_{j\in \NN_i} (x_i(t) - x_j(t)), ~ i \in \until{n}.
  \end{split}
\end{equation*}
For the controllability analysis, we consider a scenario in which some of the nodes
have the possibility to apply an additional input that fully controls its
dynamics. We call these nodes \emph{control nodes}. This turns to be the model
of a leader-follower network.
For the observability analysis, we imagine that an external processor (\emph{not} running the consensus
algorithm) collects information from some nodes in the network. We call these
nodes \emph{observation nodes}. In particular, we assume that the external
processor may read the state of each observation node.
Equivalently, we can think of one or more observation nodes, running the
consensus algorithm, that have to reconstruct the state of the network by
processing only their own state.
These two scenarios are captured by the model in equation \eqref{eq:induced_system}.

% \begin{remark}[Discrete time system]
%   The observability analysis performed in the paper can be easily extended to
%   suitable discrete time versions of the above continuous time model, see, e.g.,
%   \cite{GP-GN:10b}.
% \end{remark}

\begin{remark}[Equivalence with other problem set-ups]
  Straightforward results from linear system theory can be used to prove that
  the controllability problem studied in \cite{AR-MJ-MM-ME:09} and
  \cite{SM-ME-AB:09} and the dual observability problem studied in
  \cite{MJ-ME:07} can be equivalently formulated in our set up. \oprocend
\end{remark}

\bigskip

\subsubsection*{\bf Continuous time quantum and random walks}
Dynamic systems induced by the Laplacian of a graph appear also in dealing with
quantum and random walks \cite{HG-JW:03}. We concentrate on the quantum
counterpart of random walks, which have recently received great attention in the
area of (quantum) information theory. The general idea of quantum information
and computation is to solve common problems in information theory by using
axioms and rules derived from quantum theory. Specifically, quantum walks are a
computational variant of random walks in which the transition probability among
the states follows quantum laws as opposed to standard stochastic
laws. Formally, for quantum mechanical systems which are closed (i.e., not
interacting with the environment) and finite dimensional, one considers the
Schr\"odinger equation
\[
 i \hbar \dot{\psi}(t) = H(u(t)) \psi(t),
\]
where $\psi \in \CC^n$ is the quantum state and the Hamiltonian matrix $H(u)$ is
Hermitian and depends on a control $u(t)$. Continuous time quantum walks are
quantum systems whose dynamics is defined on a graph $G$. Specifically, the
Hamiltonian has the form
\[
  H(u) = H_0 + \sum_{j=i_1}^{i_m} e_j e_j^T u_j,
\]
where $H_0$ is the adjacency matrix or the Laplacian of a given graph $G$. The
connection with our results appears for quantum walks on a grid (or lattice) graph with $H_0$
being the grid Laplacian. The resulting dynamics is
\[
 i \hbar \dot{\psi}(t) = (L + \sum_{j=i_1}^{i_m} e_j e_j^T u_j(t)) \psi(t).
\]
In \cite{DB-DD-LH-SS-MY:11} it is shown that the controllability of the above system, expressed by
a Lie algebra rank condition, is equivalent to the controllability of the
linear system \eqref{eq:induced_system}. Our analysis is strictly related to the
line pursued in \cite{DB-DD-LH-SS-MY:11} of finding more easily verifiable graph theoretic tests.

% It is worth noting that quantum walks on a grid or lattice graph are widely
% studied in quantum information and computation ???.

 \bigskip

\subsubsection*{\bf Discretization of a class of partial differential equations}
Next, we show how the discretization of partial differential equations
containing the Laplace operator gives rise to an ordinary differential equation
whose controllability and observability can be studied by using the tools developed in the paper. 
% for the Laplacian of a grid graph.

Let $\map{f}{\real^k}{\real}$ be a twice differentiable real valued function,
then the Laplace operator of $f$ is $\Delta f := \sum_{i=1}^k \frac{\partial^2 f(x)}{\partial x^2_i}$.
% , $\Delta f$, is defined by the divergence of the
% gradient of $f$, that is 
% \[
% \Delta f = \nabla \cdot \nabla f = \sum_{i=1}^k \frac{\partial^2 f(x)}{\partial x^2_i}.
% \]
%
This operator has a key importance in several physical phenomena. In particular,
it appears in the heat and fluid flow diffusion, in wave
propagation and quantum mechanics. Specifically, the density (temperature)
fluctuations of diffusing material (heat) are described by the partial differential equation
\[
\frac{\partial\phi(x,t)}{\partial t} = D \Delta \phi(x,t) + f(x,t),
\]
where $\phi(x,t)$ is the density of the diffusing material (respectively the temperature) at location
$x\in\real^k$ and time $t$, $D$ is the diffusion coefficient (respectively the
thermal diffusivity) and $f(x,t)$ is a material (heat) source. The wave propagation, arising in acoustics,
electromagnetism and fluid dynamics, is described by the partial differential
equation
\[
\frac{\partial^2\phi(x,t)}{\partial t^2} = c^2 \Delta \phi(x,t) + f(x,t),
\]
with $\phi(x,t)$ the wave amplitude at $x\in\real^k$ and time $t$,
$c$ a constant, and $f(x,t)$ a forcing term.

% Next, we compute the discretization of the Laplacian operator over a $d$
% dimensional hyper-rectangle (i.e., the Cartesian product of $d$ compact
% intervals). For clarity of presentation, we show the calculations for a two
% dimensional domain $[x_1^{min}, x_1^{max}]\times [x_2^{min}, x_2^{max}]$ with
% $x_1^{min}$, $x_1^{max}$, $x_2^{min}$ and $x_2^{max}$ real numbers. 

If we consider a regular discretization of a $d=2$ dimensional hyper-rectangular
domain, e.g. $[x_1^{min}, x_1^{max}]\times [x_2^{min}, x_2^{max}]$,
for $i\in \until{n_1}$ and $j\in \until{n_2}$, $x_{[i,j]} = [x_1^{min} + i h,
x_2^{min} + j h]$, where $h$ is the discretization step, then 
the discretization of the Laplacian operator becomes
\[
\Delta(\phi(x_{[i,j]},t)) \approx  [- L \, \phi^d(t)]_{[i,j]},
\]
where $L$ is the Laplacian of an $n_1\times n_2$ grid graph and the vector
$\phi^d(t)\in\real^{n_1\cdot n_2}$ has components ${\phi^d}_{[i,j]}(t) = \phi(x_{[i,j]},t)$.
%
% Next, we compute the discretization of the Laplacian operator over a $d$
% dimensional hyper-rectangle (i.e., the Cartesian product of $d$ compact
% intervals). For clarity of presentation, we show the calculations for a two
% dimensional domain $[x_1^{min}, x_1^{max}]\times [x_2^{min}, x_2^{max}]$ with
% $x_1^{min}$, $x_1^{max}$, $x_2^{min}$ and $x_2^{max}$ real numbers. We
% discretize the domain, so that for example, for $i\in \until{n_1}$ and $j\in \until{n_2}$, $x_{[i,j]} =
% [x_1^{min} + i h, x_2^{min} + j h]$, where $h$ is the discretization step.
%
% The approximation of the second derivative along the two directions at $x_{[i,j]}$ is given by
% \[
% \frac{\partial^2 \phi(x_{[i,j]},t)}{\partial x_1^2} \approx \frac{\phi(x_{[i+1,j]},t)- 2\phi(x_{[i,j]},t) + \phi(x_{[i-1,j]},t)}{h^2}
% \]
% \[
% \frac{\partial^2 \phi(x_{[i,j]},t)}{\partial x_2^2} \approx \frac{\phi(x_{[i,j+1]},t)- 2\phi(x_{[i,j]},t) + \phi(x_{[i,j-1]},t)}{h^2}
% \]
% so that, summing the two terms, the discretization of the Laplacian operator becomes
% \[
% \Delta(\phi(x_{[i,j]},t)) = \frac{\partial^2 \phi(x_{[i,j]},t)}{\partial x^2_1} +
% \frac{\partial^2 \phi(x_{[i,j]},t)}{\partial x^2_2} \approx  [- L \, \phi^d(t)]_{[i,j]},
% \]
% where $L$ is the Laplacian of an $n_1\times n_2$ grid graph and the vector
% $\phi^d(t)\in\real^{n_1\cdot n_2}$ has components ${\phi^d}_{[i,j]}(t) = \phi(x_{[i,j]},t)$.
%
With this approximation in hands, the discretized versions of the partial
differential equations above are ordinary differential equations
with the same structure as in \eqref{eq:induced_system}.

\section{Controllability and observability of simple eigenvalues in grid graphs}
\label{sec:simple_grids}
In this section we characterize the controllability and observability properties of
the simple eigenvalues of the grid, namely the eigenvalues of multiplicity one.

\subsection{Laplacian eigenstructure of cartesian-product graphs}
An important property of graphs obtained as the cartesian product of other
graphs is that the Laplacian can be obtained from the Laplacian of their
constitutive graphs by using the Kronecker product of two matrices, see
\cite{RM:98}. Given two matrices $A\in\real^{d\times d}$ and $B\in\real^{l
  \times l}$, with $[A]_{ij}:= a_{ij}$, their Kronecker product $A \otimes B \in
\real^{dl\times dl}$ is defined as
\[
A \otimes B =
\begin{bmatrix}
  a_{11} B  &a_{12} B  &\ldots &a_{1d} B\\
  a_{21} B  &a_{22} B  &\ldots &a_{2d} B\\
  \vdots    &    \cdots&       &  \vdots\\
  a_{d1} B  &a_{12} B  &\ldots &a_{dd} B\\
\end{bmatrix},
\]
and their Kronecker sum as
\[
A \oplus B = A \otimes I_l + I_d \otimes B.
\]
Given the cartesian product of the graphs $G_1, \ldots, G_d$ with Laplacian
matrices $L_1, \ldots, L_d$, the Laplacian $L_\Box$ of $G_1 \,\Box\, \ldots
\,\Box\, G_d$ is given by
\[
L_\Box = L_1 \oplus \ldots \oplus L_d.
\]
This structure on the Laplacian induces a structure also on its eigenvalues and
eigenvectors. We state it in the next lemma, see \cite{RM:98}.

\begin{lemma}[Laplacian eigenstructure of cartesian product graphs]
  \label{lmm:eigstruc_cartesian_prod}
  Let $G_1, \ldots, G_d$ be $d\in\natural$ undirected graphs and $G = G_1\Box
  \ldots \Box G_d$ their cartesian product. Let $\lambda_1^\kappa, \ldots,
  \lambda_{n_\kappa}^\kappa$ be the Laplacian eigenvalues of the graphs
  $G_\kappa$ and $v_1^\kappa, \ldots, v_{n_\kappa}^\kappa$ the corresponding
  eigenvectors for $\kappa\in\until{d}$. The Laplacian eigenvalues and the
  corresponding eigenvectors of $G$ are
  \[
  \begin{split}
    \lambda^1_{i_1} + \lambda^2_{i_2}& + \ldots + \lambda^d_{i_d} \qquad
    \text{and} \qquad  v^1_{i_1} \otimes v^2_{i_2} \otimes \ldots \otimes
    v^d_{i_d}\\
  \end{split}
  \]
  % and
  % \[
  % \begin{split}
  %   v_{i_1} \otimes v_{i_2} \otimes& \ldots \otimes v_{i_d},\\
  % \end{split}
  % \]
  for $i_1 \in\until{n_1}, \ldots, i_d \in\until{n_d}$. \oprocend
\end{lemma}

We are now ready to define a simple cartesian product graph.
\begin{definition}[Simple cartesian-product graphs]
  Let $G$ and $G'$ be two undirected graphs and let $\{\lambda_1, \ldots,
  \lambda_{k}\}$ and $\{\lambda'_1, \ldots, \lambda'_{\kappa}\}$ be the sets of
  distinct eigenvalues among all the Laplacian eigenvalues of respectively $G$
  and $G'$. We say that the graph $G_{\Box} = G \, \Box \, G'$ is \emph{simple} if the
  set $\{\lambda_i + \lambda'_\alpha \;|\; i\in\until{k},
  \alpha\in\until{\kappa}\}$ contains only distinct eigenvalues. \oprocend
\end{definition}

Using the associative property of the cartesian product the definition easily
generalizes to the product of more than two graphs.

% Next proposition gives a necessary (but not sufficient) condition for a graph
% to be simple.

% \begin{proposition}[Necessary condition for grids to be simple]
%   dimensions coprime!
% \end{proposition}

% {\color{red} REMARKS maybe give a necessary condition to have a simple graph
% (n and m coprime) ADD a remark saying that with the rules for the simple
% graphs we can give conditions for observability of all the eigenvalues except
% for the multiple ones.  }

\subsection{Controllability and observability of the simple eigenvalues}
\label{subsec:grid}
% Before stating our results we recall that for a $d$-dimensional grid
% graph $G = P_1 \Box \ldots \Box P_d$, we denote $i = [(i)_1, \ldots, (i)_d]$ a node
% of $G$, where the component $(i)_\kappa$ identifies the position of the node on
% the $\kappa$th path. Also, given a Laplacian eigenvector of the G, $w \in
% \real^{n_1 \ldots n_d}$, we say ``the component $[(i)_1, \ldots, (i)_d]$ of w''
% meaning ``the component $(i)_1 \cdot (n_1 \cdot n_2 \cdot \ldots \cdot n_d) +
% (i)_2 \cdot (n_2 \cdot \ldots \cdot n_d) \ldots + (i)_d$ of w''.

We start with a lemma that relates the controllability (observability) of simple eigenvalues of a
grid from a single node to the controllability (observability) of its constitutive paths.

\begin{lemma}
  \label{lmm:unobs_siso}
  Let $P_1, \ldots, P_d$, $d\in\natural$, be path graphs of length respectively
  $n_1,\ldots, n_d$ and let $G = P_1\Box \ldots \Box P_d$. Any simple eigenvalue
  $\lambda = \lambda_1 + \ldots + \lambda_d$ of the grid graph $G$ is not
  controllable (observable) from a node $[(i)_1, \ldots, (i)_d]$, $(i)_1\in\until{n_1}$,
  $\ldots$ , $(i)_d\in\until{n_d}$, if and only if there exists $\ell \in
  \until{d}$ such that the eigenvalue $\lambda_\ell$ of $P_\ell$ is not
  controllable (observable) from $(i)_\ell$.
\end{lemma}

\begin{proof}
  From Lemma~\ref{lmm:PBH_eigvec} the
  eigenvalue $\lambda$ of the grid graph is not controllable (observable) from a node $[(i)_1,
  \ldots, (i)_d]$ if and only if a Laplacian eigenvector $w\in\real^{n_1 \ldots
    n_d}$ of $\lambda$ has zero $[(i)_1, \ldots, (i)_d]$ component.
  Using Lemma~\ref{lmm:eigstruc_cartesian_prod} and the assumption that $\lambda$ is a simple
  eigenvalue, any eigenvector $w$ of $\lambda$ can be written as $w = \rho \, v_1 \otimes \ldots
  \otimes v_d$, where
  $\rho\in\real$, $\rho\neq0$, and $v_1, \ldots v_d$ are
   eigenvectors of the constitutive paths.
   Using the structure of the Kronecker product of $d$ vectors, the $[(i)_1,
   \ldots, (i)_d]$ component of $w$ is zero if and only if at least one
   $v_\ell$, $\ell\in\until{d}$, has zero $(i)_\ell$ component. Indeed,
   $(w)_{[(i)_1, \ldots, (i)_d]} = \rho (v_1)_{i_1}\cdot \ldots \cdot
   (v_d)_{i_d}$ which is zero if and only if there exists $\ell\in\until{d}$
   such that $(v_\ell)_{i_\ell}=0$. From Lemma~\ref{lmm:PBH_eigvec}
   $(v_\ell)_{i_\ell}=0$ if and only if the eigenvalue $\lambda_{\ell}$ of
   $P_\ell$ is not controllable (observable) from $(i)_\ell$, thus concluding the proof.
\end{proof}

From the previous lemma and Lemma~\ref{lmm:PBH_eigvec} the next proposition
follows straight.
\begin{proposition}
  A simple grid $G=P_1\Box\ldots\Box P_d$ is controllable (observable) from a node $i=[(i)_1, \ldots, (i)_d]$ if and only
  if each path $P_\ell$ is controllable (observable) from node $(i)_\ell$.\oprocend
\end{proposition}

Using the property that any path graph is controllable (observable) from each external node,
\cite{GP-GN:12}, the next corollary follows. 
\begin{corollary}
  Any simple eigenvalue of a grid graph is controllable (observable) from any
  corner node.\oprocend
\end{corollary}

We are now ready to characterize the controllability (observability) of simple grid graphs.

\begin{theorem}[Simple grid controllability (observability)]
  \label{thm:main_thm_simple_grid}
  Let $P_1, \ldots, P_d$ be $d$ path graphs of length respectively $n_1,\ldots,
  n_d$ and let $G = P_1\Box \ldots \Box P_d$ be a simple grid. The following
  statements hold.
  \begin{enumerate}
  \item The grid graph $G$ is not controllable (observable) from a node $i=[(i)_1, \ldots,
    (i)_d]$, $(i)_1\in\until{n_1}$, $\ldots$ , $(i)_d\in\until{n_d}$, if and
    only if there exists $\ell \in \until{d}$ such that
    \begin{equation}
      \big(n_\ell-(i)_\ell\big)\; \eqmod{p} \;\big((i)_\ell-1\big),
      \label{eq:unobs_siso}
    \end{equation}
    for some odd prime $p$ dividing $n_\ell$;
  \item for any direction $\ell\in\until{d}$ of $G$ the following holds. For
    each odd prime factor $p$ of $n_\ell$, the grid is not controllable (observable) from the
    set of nodes $I_s^\ell = \{ i=[(i)_1, \ldots, (i)_d]\in I \;|\;  (i)_\ell = j \,
  p-\frac{p-1}{2}, {j\in\until{\frac{n}{p}}}, \;\text{and}\; (i)_1, \ldots,(i)_{\ell-1},
    (i)_{\ell+1}, \ldots, (i)_d \;\text{arbitrary}\}$, with the following uncontrollable (unobservable)
    eigenvalues
    \begin{equation}
      \begin{split}
        \lambda_{\nu,\ell} = 2-2\cos\left((2 \nu-1) \frac{\pi}{{p}}\right) +
        \lambda_1 + \ldots+\lambda_{\ell-1} + \lambda_{\ell+1} + \ldots
        +\lambda_d,
      \end{split}
      \label{eq:unobs_eigs_grid}
    \end{equation}
    and uncontrollable (unobservable) eigenvectors
    \begin{equation}\label{eq:unobs_eigvec_grid}
      \begin{split}
        w_{\nu,\ell}=u_1\otimes u_2\otimes .. \otimes u_{\ell-1}\otimes
        V_{\nu,\ell} \otimes  u_{\ell+1} ..  \otimes  u_{d}, \quad
         V_{\nu,\ell} \;\; \text{as in \eqref{eq:unobs_eigvecs_path}}
      \end{split}
    \end{equation}
    with $\nu\in\until{(p-1)/2}$, and $\lambda_\mu$, respectively $u_\mu$,
    $\mu\neq\ell$, any arbitrary eigenvalue, respectively eigenvector, of
    $P_\mu$;
  \item if a node $i=[(i)_1, \ldots, (i)_d]$ satisfies \eqref{eq:unobs_siso} for
    $r\leq d$ distinct directions and, in each direction $\ell$, for $k_\ell\leq
    n_\ell$ distinct prime factors, then the set of uncontrollable (unobservable) eigenvalues
     from node $i$ is the union of eigenvalues with the following structure
    \[
      \begin{split}
        \lambda_{\ell} = \bar{\lambda} +
        \lambda_1 + \ldots+\lambda_{\ell-1} + \lambda_{\ell+1} + \ldots
        +\lambda_d,
      \end{split}
    \]
    where each $\bar{\lambda}$ is an uncontrollable (unobservable) eigenvalue of
    path $P_\ell$ from $(i)_\ell$ and has the structure given in
    Theorem~\ref{thm:main_thm_path} (iv) and
    Remark~\ref{rmk:general_main_thm_path} in Appendix. The corresponding
    uncontrollable (unobservable) eigenvectors can be built according to equation
    \eqref{eq:unobs_eigvec_grid}.
    % defined as in \eqref{eq:unobs_eigs_grid} (respectively in
    % \eqref{eq:unobs_eigvec_grid}) for each direction $\ell$ and each prime
    % factor are unobservable from $i$.
  \oprocend
  \end{enumerate}
\end{theorem}
      
\begin{proof}
  The proof of statement (i) follows straight by combining the result of
  Lemma~\ref{lmm:unobs_siso} with the result in (i) of
  Theorem~\ref{thm:main_thm_path}.

  To prove statement (ii), we start observing that, for each $\ell\in\until{d}$,
  the set of nodes $I_s^\ell = \{ (i)_\ell \,
  p-\frac{p-1}{2}\}_{(i)_\ell\in\until{\frac{n}{p}}}$ is the set of all nodes
  satisfying condition in \eqref{eq:unobs_siso} for a given $p$ in the path
  $P_\ell$. Using Theorem~\ref{thm:main_thm_path}, we have that the uncontrollable
  (unobservable) eigenvalues of $P_\ell$ from this set of nodes have the form in
  \eqref{eq:unobs_eigs_path}. Now, according to
  Lemma~\ref{lmm:eigstruc_cartesian_prod} all the $\lambda_{\nu,\ell}$ as in
  \eqref{eq:unobs_eigs_grid} are eigenvalues of the grid. Also, the
  corresponding eigenvectors are the ones given in
  \eqref{eq:unobs_eigvec_grid}. Using the result in Lemma~\ref{lmm:unobs_siso},
  these eigenvectors have a zero in position $[(j)_1, \ldots, (j)_{\ell-1},
  (i)_\ell,(j)_{\ell+1}\ldots (j)_{d}]$ with $(i)_\ell$ satisfying
  \eqref{eq:unobs_siso} and $(j)_1, \ldots, (j)_{\ell-1}, (j)_{\ell+1}\ldots
  (j)_{d}$ arbitrary.
  % NB: maybe we need to say that these are the only ones in common to all of
  % them.
  To conclude the proof, we show that these are the only uncontrollable (unobservable)
  eigenvectors. To do that, we show that there exist nodes in $I_s^\ell$ for which the
  only zero component is $(i)_\ell$. For example, take any node with
  index $[1, \ldots,1,(i)_\ell,1\ldots 1]$ with $(i)_\ell$ satisfying
  \eqref{eq:unobs_siso}. Since any path is controllable (observable) from the first node, the
  proof follows.
  
  Statement (iii) follows straight by combining the results in the previous two
  points with Theorem~\ref{thm:main_thm_path}.
\end{proof}

%% Notice that, if relation (\ref{eq:condizione_nonobs_siso}) is satisfied for a
%% non-prime $p=p_1 p_2 .. p_k$ then it must be satisfied for each prime
%% $p_i$. In this case the unobservable eigenvalues and directions are the union
%% of the sets computed for each $p_i$.

Next, we show, on the basis of the results in Theorem~\ref{thm:main_thm_simple_grid},
how to check the controllability (observability) of a simple grid from a given set of nodes or,
equivalently, how to construct a set of control (observation) nodes such that the grid is
controllable (observable). For the sake of clarity we present the procedure for a two
dimensional grid ($d=2$), however the procedure can be easily generalized to
higher dimensions.

First, we introduce some notation. Given two sets $X = \union_{\nu=1}^k
[x_{1,\nu},x_{2,\nu}]$ and $Y = \union_{\nu=1}^l [y_{1,\mu},y_{2,\mu}]$ with
$[x_{1,\nu},x_{2,\nu}] \in \real\times\real$ and $[y_{1,\nu},y_{2,\nu}] \in
\real\times\real$, we say that $X \intersection Y \neq \emptyset$ if there
exists $[x_{1,\bar{\nu}},x_{2,\bar{\nu}}]\in X$ and
$[y_{1,\bar{\mu}},y_{2,\bar{\mu}}]\in Y$ such that
$[x_{1,\bar{\nu}},x_{2,\bar{\nu}}]=[y_{1,\bar{\mu}},y_{2,\bar{\mu}}]$,
i.e. $x_{1,\bar{\nu}}=y_{1,\bar{\mu}}$ and $x_{2,\bar{\nu}}=y_{2,\bar{\mu}}$.

Consider a two dimensional simple grid graph $G = P_1\Box P_2$ with $P_1$ and
$P_2$ of dimension $n_1$ and $n_2$ respectively. Let, for each $\ell\in\{1,2\}$,
$n_\ell = 2^{n_{\ell,0}} \prod_{\nu =1}^{k_\ell} p_{\ell,\nu}$ be a prime number
factorization for some $k_\ell\in \natural$ and odd prime numbers $p_{\ell,1},
\ldots, p_{\ell,k_\ell}$. Let $I_s = \{i_1, \ldots, i_m\}$ be a set of
control (observation) nodes with $i_\alpha = [(i_\alpha)_1, (i_\alpha)_2]$,
$\alpha\in\until{m}$. Now, we construct $m$ sets $O_1, \ldots, O_m$ that will be
used to define a simple rule for controllability (observability). For the sake of clarity we
provide the rule to construct a set $O_\alpha$ for a specific case. The general
case can be easily deduced from the example. Suppose that $i_j$ satisfies
condition \eqref{eq:unobs_siso} for $p_{1,1}$ and $p_{1,2}$ along direction $1$
and for $p_{2,3}$ along direction $2$. Now, define the set $O_\alpha$ as
follows.
\[
\begin{split}
  O_\alpha = [p_{1,1}, p_{2,1}] \union \ldots \union [p_{1,1}, p_{2,k_2}]\union [p_{1,2}, p_{2,1}] \union %\ldots\\
  \ldots \union [p_{1,2}, p_{2,k_2}] \union [p_{1,1}, p_{2,3}] \union \ldots
  \union[p_{1,k_1}, p_{2,3}].
\end{split}
\]
We call $O_1, \ldots, O_m$ a controllability (observability) partition of the set $I_s$.

The following proposition gives an easily implementable test for
controllability (observability). 
% The proof is straightforward and thus omitted.
%
\begin{proposition}[Controllability (observability) test]
  Let $G = P_1 \Box P_2$ be a simple grid graph and $I_s$ a set of control (observation)
  nodes with controllability (observability) partition $O_1, \ldots, O_m$. Then $G$
  is controllable (observable) from $I_s$ if and only if $O_1 \intersection \ldots
  \intersection O_m = \emptyset$. \oprocend
\end{proposition}
\begin{proof}
  The grid is controllable (observable) from the set $I_s$ if and only if the
  intersection of the sets of uncontrollable (unobservable) eigenvalues from each
  node is empty (equivalently if the intersection of the uncontrollable
  (unobservable) subspaces is the zero vector). Using statement (iii) of
  Theorem~\ref{thm:main_thm_simple_grid} the uncontrollable (unobservable)
  eigenvalues from each node are the ones in equation
  \eqref{eq:unobs_eigs_grid}.
  Let $\lambda = \lambda_1 + \lambda_2$, with $\lambda_1$ eigenvalue of $P_1$
  and $\lambda_2$ eigenvalue of $P_2$, be a common eigenvalue to all the nodes
  in $I_s$.
  Using the property in (iii), a node $i = [(i)_1 (i)_2] \in I_s$ can have that
  uncontrollable (unobservable) eigenvalue either because it is uncontrollable
  (unobservable) along one of the two paths or along both. Equivalently either
  $\lambda_1$ is an uncontrollable (unobservable) eigenvalue for $P_1$ from node
  $(i)_1$ and/or $\lambda_2$ is an uncontrollable (unobservable) eigenvalue for
  $P_2$ from node $(i)_2$.

  Now, using the result in statement (ii) of Theorem~\ref{thm:main_thm_path},
  all the nodes in a control (observation) set $I_{o1}$ are uncontrollable
  (unobservable) if and only if all of them belong to a set $I_1^p=\{ \ell p -
  \frac{p-1}{2}\}_{\ell\in\until{\frac{n_1}{p}}}$ for some factor $p$ of $n_1$
  ($n_1$ being the length of $P_1$). The same argument holds for a control
  (observation) set $I_{o2}$ on the path $P_2$.
  This implies that the controllability (observability) on a path $P_i$ can be
  studied by checking if all the nodes in the control (observation) set $I_{oi}$
  share a common prime factor.

  Now, each point in the set $O_\alpha$, $\alpha\in\until{m}$, is of the form
  $[p_{1}, p_{2}]$ where $p_1$ and $p_2$ are prime factors of $n_1$ and $n_2$
  (the lengths of $P_1$ and $P_2$) respectively and at least one of the prime
  factors, say $p_1$, is such that $(i_\alpha)_1 \in I_1^{p_1}=\{ \ell p_1 -
  \frac{p_1-1}{2}\}_{\ell\in\until{\frac{n_1}{p_1}}}$. Thus, each element in the
  set $O_\alpha$ correspond to set of eigenvalues $\Lambda_\alpha =
  \big\{\lambda_{\nu_1,\nu_2}\in\real \;|\; \lambda_{\nu_1,\nu_2} =2-2\cos\left((2
    \nu_1-1) \frac{\pi}{{p_1}}\right) + 2-2\cos\left((2 \nu_2-1)
    \frac{\pi}{{p_2}}\right), \nu_1\in \until{\frac{n_1}{p_1}} \; \text{and} \;
  \nu_2\in \until{\frac{n_2}{p_2}}\big\}$ (with $(i_\alpha)_1 \in I_1^{p_1}$ or/and
  $(i_\alpha)_2 \in I_1^{p_2}$). The proof follows by observing that the
  controllability (observability) condition is that the intersection of the sets $\Lambda_\alpha$
  be empty.
\end{proof}

The following examples can be easily explained by using the proposition
above. If at least one of the control (observation) nodes, say $i_1$, is
controllable (observable) in any direction, then the grid is controllable
(observable). Indeed, the set $O_1$ will be empty. If all $n_\ell$,
$\ell\in\until{d}$, are prime, then the grid is controllable (observable) if and
only if one of the control (observation) nodes is controllable (observable) in any
direction. Indeed, any $O_\alpha$, $\alpha\in\until{m}$, can be either $O_\alpha
= \{[n_1, \ldots, n_d]\}$ or $O_\alpha = \emptyset$.

Next, we show a graphical interpretation of the controllability (observability) test based on the
controllability (observability) partition. We present it through an example. In
Figure~\ref{fig:grid_7x15} we show a two dimensional grid of length $7\times
15$. It can be easily tested that this grid is simple. In each direction
$\ell\in\{1,2\}$, for each prime factor of $n_l$ we associate a unique symbol to
the rows (columns) of nodes that satisfy \eqref{eq:unobs_siso} for that prime
number (in that direction). In particular, for the grid in
Figure~\ref{fig:grid_7x15}, we associate a cross to the columns satisfying
\eqref{eq:unobs_siso} for the prime factor $5$ of $15$, a triangle to the columns
satisfying \eqref{eq:unobs_siso} for the prime factor $3$ of $15$, and a
pentagon to the unique row satisfying \eqref{eq:unobs_siso} for the prime number
$7$.

\begin{figure}[htbp]
  \centering
  \includegraphics[height=.4\linewidth]{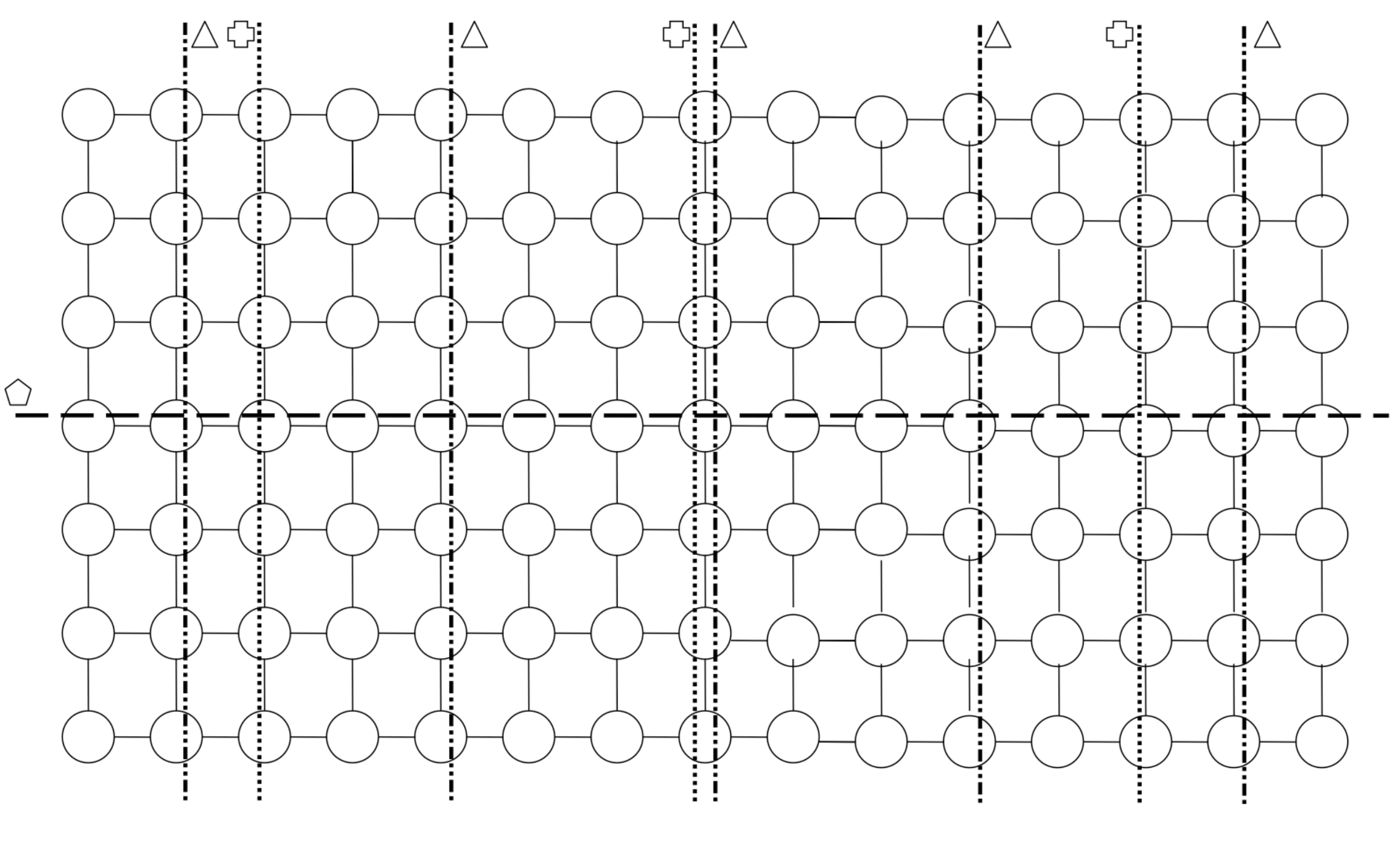}%
  % \vspace{-7 mm}
  % \vspace{-7 mm}
  \caption{Controllability (observability) partition for a $7\times 15$ grid graph.}
  \label{fig:grid_7x15}
\end{figure}

Clearly, all the nodes that are not crossed by any line are controllable
(observable). Also, a subset of nodes from which the graph is controllable
(observable) can be easily constructed by suitably combining the different
symbols. Equivalently, given a set of control (observation) nodes, testing the
associated symbols easily gives the controllability (observability) property
from the given set of nodes. For example, from the pair of nodes $i_1=[1,2]$ and
$i_2=[4,1]$ the grid is not controllable (observable). Indeed, the partitions
are $O_1 = [3,7]$ and $O_2 = [3,7] \union [5,7]$ whose intersection is
$[3,7]$. The uncontrollable (unobservable) eigenvalues are, according to
Theorem~\ref{thm:main_thm_simple_grid}, $1+(2-2\cos\frac{\pi}{7})$,
$1+(2-2\cos\frac{3 \pi}{7})$ and $1+(2-2\cos\frac{5 \pi}{7})$.
Following the same logic the grid is controllable (observable) form the set $[1,2]$ and
$[1,3]$, but it is not controllable (observable) from  $[1,2]$,
$[1,8]$ and $[4,1]$ (and from any subset of them). 
We let the reader play with it and have fun.

% \clearpage
\section{Eigenstructure of general grid graphs}
\label{sec:general_grids}
In order to characterize the controllability and observability of general grid
graphs we need to exploit their eigenstructure. Indeed, the main difference with
respect to the simple case analysis relies in the structure of the uncontrollable
(unobservable) eigenvectors. While for simple grids they can be always written
as the Kronecker product of two eigenvectors of the path (because the
eigenvalues are all simple), this property does not hold for the eigenvectors of
non-simple grids. Thus, the controllability (observability) analysis can not be
performed by simply looking at how the zeros of the path eigenvectors propagate
into the grid. Indeed, this analysis provides only necessary conditions for
controllability (observability).

This section will be organized as follows. First, we characterize symmetries in
the structure of the grid eigenvectors. This analysis allows us to recognize the
components of the eigenvectors that have to be equal. Second, we provide
conditions to show what are all and only the components that are zero when a
given component is forced to zero. Thus, with this results in hand, we are able
to provide \emph{necessary and sufficient} conditions for controllability (observability).

We begin by characterizing symmetries of the path eigenvectors and then, using
these results, we characterize symmetries of the grid eigenvectors by suitable
grid partitions.
For the sake of clarity we provide the analysis and results for two dimensional
grids ($d=2$). The results for higher dimensions are based on the
  same arguments and will be discussed in a remark.

% \clearpage
% \section{Alegebraic structure of path graphs}
% \label{sec:path_results}
% In this section we show ...

\subsection{Symmetries of the path Laplacian eigenvectors}
\label{subsec:path_results}
We provide results on the structure and symmetries of the
Laplacian eigenvectors of a path graph. 
% that will be used to study the observability of grid graphs via suitable, symmetry based, subgrid partitions.
%
The next lemma characterizes the symmetry of the path Laplacian eigenvectors.
% a nice simple result on the eigenvectors of the path Laplacian which however
% will play a key role in the analysis of the grid observability.
%
\begin{lemma}[Symmetries of the path Laplacian eigenvectors]
  \label{lem:sym_path_eigvec}
  Any eigenvector $v$ of the Laplacian of a path graph satisfies either $v = \Pi
  v$ or $v = -\Pi v$, with $\Pi$ the usual permutation matrix.
\end{lemma}

\begin{proof}
  Let $L \in \real^{n\times n}$ be the Laplacian of the path. Straightforward
  calculations show that $L$ satisfies $L = \Pi L \Pi$. Now, let $v$ be a
  Laplacian eigenvector, then
  % \[
  $\Pi L \Pi v = \lambda v$.
  % \]
  Multiplying both sides by $\Pi$ (and remembering that $\Pi^2 = I$), we
  get %\[
  $L \Pi v = \lambda \Pi v$,
  % \]
  so that $\Pi v$ is also an eigenvector of $L$ associated to the eigenvalue
  $\lambda$. Since any eigenvalue of $L$ has multiplicity one, it must hold
  % \[
  $v = \alpha \Pi v$,
  % \]
  for some nonzero $\alpha \in \real$. Using the fact that the linear map $\Pi$
  is an isometry (i.e. it preserves the norm), $||\Pi v || = || v||$, it follows
  straight that either $\alpha = 1$ or $\alpha = -1$, which concludes the proof.
\end{proof}

In the rest of the paper we will denote $S^+$ (respectively $S^-$) the set of
vectors satisfying $v=\Pi v$ (respectively $v=-\Pi v$). An important property of
$S^+$ and $S^-$ is that one is the orthogonal complement of the other,
i.e. $(S^+)^\perp = S^{-}$.

The next lemma relates the eigenstructure of a given path $P$ to the eigenstructure
of any path with length multiple of the length of $P$.
\begin{lemma}[Laplacian eigenstructure of $P_{n}$ and $P_{kn}$]
  \label{lmm:sym_path_eigvec_multiple}
  Let $\lambda_1, \ldots, \lambda_{n}$ be the eigenvalues of the Laplacian $L_n$
  of a path $P_n$ of length $n$ and $v_1, \ldots, v_{n}$ the corresponding
  eigenvectors. Then any path $P_{kn}$ of length $k n$, for some $k \in \natural$, with
  Laplacian matrix $L_{k n}$ satisfies:
  \begin{enumerate}
  \item $\lambda_1, \ldots, \lambda_{n}$ are eigenvalues of $L_{k n}$;
  \item each eigenvector $w_i\in \real^{k n}$ of $L_{k n}$ associated to
    $\lambda_i$, $i\in \until{n}$, has the form 
\[
w_i =
    \begin{bmatrix}
      v_i^T &
      (\Pi v_i)^T &
      v_i^T&
      \ldots
    \end{bmatrix}^T.
\]
  %   \[
  %   w_i =
  %   \begin{bmatrix}
  %     v_i\\
  %     \Pi v_i\\
  %     v_i\\
  %     \vdots\\
  %   \end{bmatrix}.
  %   \]
   \end{enumerate}
\end{lemma}

\begin{proof}
  For the sake of clarity we prove the statement for $k=3$, but the proof for
  the general case is easily generalizable.
  The Laplacian $L_{kn}$ can be written in terms of $L_n$, whose structure is
  given in Appendix, as
  \[
  L_{k n} =
  \begin{bmatrix}
    L_n + e_n e_n^T & -e_n e_1^T & 0\\[1.2ex]
    -e_1 e_n^T & L_n + e_1 e_1^T + e_n e_n^T & -e_n e_1^T\\[1.2ex]
    0 & -e_1 e_n^T & L_n + e_1 e_1^T
  \end{bmatrix}.
  \]
  Now, let us write the eigenvector $w_i$ associated to $\lambda_i$ as $w_i =
  \begin{bmatrix}
    v_a^T & v_b^T & v_c^T
  \end{bmatrix}^T$,
  % \[
  % w_i =
  %	\begin{bmatrix}
  %   v_a\\
  %   v_b\\
  %   v_c
  %	\end{bmatrix},
  %	\]
  with $v_a$, $v_b$ and $v_c$ in $\real^n$. Thus $w_i$ satisfies
  \[
  L_{kn} w_i =
  \begin{bmatrix}
    L_n v_a + e_n \big( (v_a)_n - (v_b)_1 \big )\\[1.2ex]
    L_n v_b + e_1 \big( (v_b)_1 - (v_a)_n \big ) + e_n \big( (v_b)_n - (v_c)_1 \big )\\[1.2ex]
    L_n v_c + e_1 \big( (v_c)_1 - (v_b)_n \big )
  \end{bmatrix}
  \]
  Now let us take $v_a = v_i$, $v_b = \Pi v_i$ and $v_c = v_i$, then
  \[
  \begin{split}
    L_{kn} w_i &= \lambda_i w_i +
    \begin{bmatrix}
      e_n \big( (v_i)_n - (\Pi v_i)_1 \big )\\[1.2ex]
      e_1 \big( (\Pi v_i)_1 - (v_i)_n \big ) + e_n \big( (\Pi v_i)_n - (v_i)_1 \big )\\[1.2ex]
      e_1 \big( (v_i)_1 - (\Pi v_i)_n \big )
    \end{bmatrix}
    = \lambda_i w_i.
  \end{split}
  \]
  Last equality follows by the fact that $(\Pi v_i)_1 = (v_i)_n$ and $(\Pi
  v_i)_n = (v_i)_1$ (in general $(\Pi v_i)_\ell = (v_i)_{n-\ell+1}$ for any
  $\ell \in \until{n}$).
\end{proof}

Exploiting the result in the above lemma by using the result in
Lemma~\ref{lem:sym_path_eigvec}, it follows easily that $w_i =
\begin{bmatrix}
  v_i^T & v_i^T & v_i^T &\ldots
\end{bmatrix}^T$ for $v_i = \Pi v_i$ (and thus $w_i = \Pi w_i$) and $w_i =~
\begin{bmatrix}
  v_i^T & -v_i^T & v_i^T &\ldots
\end{bmatrix}^T$ for $v_i = -\Pi v_i$ (and thus $w_i = -\Pi w_i$).

% \clearpage
\subsection{Symmetries of the grid eigenvectors}
Next, we provide tools to recognize symmetries in the grid eigenvectors, based
on the graph structure, which will play a key role in the controllability (observability)
analysis.

Without loss of generality, let $\lambda = \lambda_{1,1} + \lambda_{1,2} =
\ldots = \lambda_{\mu,1} + \lambda_{\mu,2}$ be an eigenvalue of geometric
multiplicity $\mu\in\natural$, with $\lambda_{1,1}, \ldots, \lambda_{\mu,1}$
(respectively $\lambda_{1,2}, \ldots, \lambda_{\mu,2}$) eigenvalues of $P_1$
(respectively $P_2$) and corresponding eigenvectors $v_1, \ldots v_\mu$
(respectively $w_{1}, \ldots, w_\mu$). The corresponding eigenspace
$V_{\lambda}$ is given by
\begin{equation}
  V_{\lambda} = \{ v \in \real^{n_1 \cdot n_2} | v = \sum_{i=1}^\mu \alpha_i (v_{i} \otimes w_{i}), \alpha_i \in \real \}.
  \label{eq:V_lambda}
\end{equation}

As mentioned at the beginning of this section, it is worth noting that the
eigenvectors in $V_\lambda$ do not necessarily have the structure of a Kronecker
product of two eigenvectors (the set of vectors expressed as Kronecker product
is not closed under linear combination). For this reason, in order identify all
and only the zero components of these eigenvectors, we need to characterize
their structure.

\begin{remark}
  For each node $[\nu,\ell]$ such that the paths $P_1$ and $P_2$ are controllable
  (observable) from $\nu$ and $\ell$ respectively, all the basis eigenvectors of
  $V_\lambda$ have nonzero $[\nu,\ell]$ component. \oprocend
\end{remark}

Before stating the main results of this section, we need to introduce some
useful notation.  Given a path $P_{n}$ of length $n\in\natural$, for
$\{i,j\}\subset \until{n}$, $i<j$, we denote $P_{i:j}$ the sub-path of $P_n$
with node set $\fromto{i}{j}$ (e.g., $P_{2:4}$ is the sub-path with node set
\{2,3,4\}).
Let $G = P_{l\cdot n_1} \Box P_{m\cdot n_2}$ with $P_{l\cdot n_1}$ of dimension
$l \cdot n_1$ and $P_{m\cdot n_2}$ of dimension $m \cdot n_2$. We call $G_{ij} =
P_{((i-1)n_1+1):(i n_1)} \Box P_{((j-1)n_2+1):(j n_2)}$, for $i\in\until{l}$ and
$j\in\until{m}$, an $n_1 \times n_2$ sub-grid or a \emph{brick} of $G$, see
Figure~\ref{fig:subgrid_partition}.

% \begin{figure}[htbp]
%   \centering
%   \includegraphics[height=.3\linewidth]{grid_7x15.pdf}%
%  %   \vspace{-7 mm}
%  %   \vspace{-7 mm}
%   \caption{Partition of a grid $G$ into sub-grids $G_{ij}$.}
%   \label{fig:grid_7x15}
% \end{figure}

\begin{figure}[h!]
  \begin{center}

    \begin{picture}(220,130)
	
      \setlength{\unitlength}{0.75pt}
	
      \put(10,10){\line(1,0){250}}
      \put(10,40){\line(1,0){250}}
      \put(10,100){\line(1,0){250}}
      \put(10,130){\line(1,0){250}}

      \put(10,10){\line(0,1){120}}
      \put(50,10){\line(0,1){120}}
      \put(90,10){\line(0,1){120}}
      \put(220,10){\line(0,1){120}}
      \put(260,10){\line(0,1){120}}

      % \makebox[width][position]{text}
	
      \put(0,20){\mbox{$l$}}
      % \put(10,40){\line(1,0){250}}
      \put(0,70){\mbox{$\vdots$}}
      \put(0,110){\mbox{$1$}}
      \put(22,110){\mbox{$G_{11}$}}
      \put(62,110){\mbox{$G_{12}$}}
      \put(232,110){\mbox{$G_{1m}$}}
	
      \put(22,20){\mbox{$G_{\ell 1}$}}
      \put(62,20){\mbox{$G_{\ell 2}$}}
      \put(232,20){\mbox{$G_{\ell m}$}}
	
      % Scritte superiori
      \put(25,135){\mbox{$1$}}
      \put(65,135){\mbox{$2$}}
      \put(125,135){\mbox{$\ldots$}}
      \put(235,135){\mbox{$m$}}

    \end{picture}

  \end{center}

  \caption{Partition of a grid into bricks}
  \label{fig:subgrid_partition}
\end{figure}

Let $v \in \real^{l \cdot n_1 \cdot m \cdot n_2}$ be a vector of $G$, we call
the \emph{sub-vector of $v$ associated to $G_{ij}$} the vector
$v_{ij}\in\real^{n_1 \cdot n_2}$ with components $(v_{ij})_{[\nu, \,\ell]}$,
$\nu\in\fromto{1}{n_1}$ and $\ell\in\fromto{1}{n_2}$, given by $(v_{ij})_{[\nu,
  \,\ell]} = (v)_{[(i\text{-}1) n_1\text{+}\nu, \,(j\text{-}1)
  n_2\text{+}\ell]}$.
Informally, the sub-vector $v_{ij}$ of $v$ is constructed by selecting the
components of $v$ that fall into the brick $G_{ij}$.

% defined as follows
% \[
% v_{ij} =
% \begin{bmatrix}
%   v_{ij,1}\\
%   \vdots\\
%   v_{ij,n_1}
% \end{bmatrix}
% \]
% where $v_{ij, 1}, \ldots, v_{ij, n_1}$ are vectors in $\real^{n_2}$ given by
% \[
% v_{ij, \ell} =
% \begin{bmatrix}
%   (v)_{m \cdot n_2 \cdot (i \cdot n_1 + \ell) + j n_2}\\
%   \vdots\\
%   (v)_{m \cdot n_2 \cdot (i \cdot n_1 + \ell) + j n_2 + n_2 - 1}
% \end{bmatrix}
% \]

Next, given a grid $G = P_{n_1} \Box P_{n_2}$, with $P_{n_1}$ and $P_{n_2}$
paths of length $n_1$ and $n_2$ respectively, we introduce two useful operators
that flip the components of a vector $v$ associated to a grid $G$. Formally,
consider the matrices
\[
(\Pi_{n_1} \otimes I_{n_2}) =
\begin{bmatrix}
 0_{n_2\times n_2} & \ldots & I_{n_2} \\[1.2ex]
  &   \iddots &\\[1.2ex]
  I_{n_2} & \ldots & 0_{n_2\times n_2}\\[1.2ex]
\end{bmatrix}
\; \text{and} \;\;
(I_{n_1}\otimes \Pi_{n_2}) =
\begin{bmatrix}
 \Pi_{n_2} & \ldots & 0_{n_2\times n_2}\\[1.2ex]
  & \ddots &\\[1.2ex]
  0_{n_2\times n_2} & \ldots & \Pi_{n_2}\\[1.2ex]
\end{bmatrix}.
% (\Pi_{n_1} \otimes I_{n_2}) =
% \begin{bmatrix}
%   0_{n_2\times n_2} & 0_{n_2\times n_2} & \ldots & 0_{n_2\times n_2} & I_{n_2}\\[1.2ex]
%   0_{n_2\times n_2} & 0_{n_2\times n_2} & \ldots & I_{n_2} & 0_{n_2\times n_2}\\[1.2ex]
%   &  & \iddots\\[1.2ex]
%   0_{n_2\times n_2} & I_{n_2} & \ldots & 0_{n_2\times n_2} & 0_{n_2\times n_2}\\[1.2ex]
%   I_{n_2} & 0_{n_2\times n_2} & \ldots & 0_{n_2\times n_2} & 0_{n_2\times n_2}\\
% \end{bmatrix}
% \]
% and
% \[
% (I_{n_1}\otimes \Pi_{n_2}) =
% \begin{bmatrix}
%   \Pi_{n_2} & 0_{n_2\times n_2} & \ldots & 0_{n_2\times n_2} & 0_{n_2\times n_2}\\[1.2ex]
%   0_{n_2\times n_2} & \Pi_{n_2} & \ldots & 0_{n_2\times n_2} & 0_{n_2\times n_2}\\[1.2ex]
%   &  & \ddots\\[1.2ex]
%   0_{n_2\times n_2} & 0_{n_2\times n_2} & \ldots & \Pi_{n_2} & 0_{n_2\times n_2}\\[1.2ex]
%   0_{n_2\times n_2} & 0_{n_2\times n_2} & \ldots & 0_{n_2\times n_2} & \Pi_{n_2}\\
% \end{bmatrix}.
\]
These operators flip respectively the first and the second sets of components. Formally,
given a vector $v\in\real^{n_1\cdot n_2}$ associated to the grid $G$, with
components $(v)_{[\nu, \,\ell]}$, $\nu\in\until{n_1}$ and $\ell\in\until{n_2}$,
let $v_1 = (\Pi_{n_1} \otimes I_{n_2}) v$ and $v_2 = (I_{n_1}\otimes \Pi_{n_2})
v$. The vectors $v_1$ and $v_2$ are related to $v$ by
\[
(v_1)_{[\nu, \,\ell]} = (v)_{[n_1-\nu+1, \,\ell]}, \;\; \text{and} \;\;
(v_2)_{[\nu, \,\ell]} = (v)_{[\nu, \,n_2-\ell+1]},
\]
for $\nu\in\until{n_1}$ and $\ell\in\until{n_2}$.  Finally, the composition of
the two operators satisfies $(\Pi_{n_1} \otimes I_{n_2}) (I_{n_1}\otimes
\Pi_{n_2}) = (\Pi_{n_1} \otimes \Pi_{n_2})$. Thus, when applied to a vector $v$,
the composed operator flips both sets of components. That is, denoting $v_3 = (\Pi_{n_1}
\otimes \Pi_{n_2}) v$, we have
\[
(v_3)_{[\nu, \,\ell]} = (v)_{[n_1-\nu+1, \, n_2-\ell+1]},
\]
for $\nu\in\until{n_1}$ and $\ell\in\until{n_2}$.

% An interesting property of $(\Pi_{n_1} \otimes I_{n_2})$, $(I_{n_1}\otimes
% \Pi_{n_2})$ and $(\Pi_{n_1} \otimes \Pi_{n_2})$ is that they are equal to
% their inverse.

\begin{lemma}
  Let $G_0 = P_{n_1} \Box P_{n_2}$ with $P_{n_1}$ and $P_{n_2}$ paths of length
  respectively $n_1$ and $n_2$. Any eigenvalue $\lambda$ of the Laplacian $L_0$
  of $G_0$ is an eigenvalue of the Laplacian $L$ of $G = P_{l\cdot n_1} \Box
  P_{m\cdot n_2}$ for any $l\in\natural$ and $m\in\natural$.
  \label{lmm:eigs_G0_G}
\end{lemma}

\begin{proof}
  From Lemma~\ref{lmm:eigstruc_cartesian_prod}, each eigenvalue $\lambda_0$ of
  $G_0$ (respectively $\lambda$ of $G$) can be written as $\lambda_0 =
  \lambda_{01} + \lambda_{02}$ ($\lambda = \lambda_1 + \lambda_2$)
  where $\lambda_{01}$ ($\lambda_1$) is an eigenvalue of $P_{n_1}$ ($P_{l\cdot n_1}$) and
  $\lambda_{02}$ ($\lambda_2$) of $P_{n_2}$ ($P_{m\cdot n_2}$). From
  Lemma~\ref{lmm:sym_path_eigvec_multiple} all the eigenvalues of $P_{n_1}$
  ($P_{n_2}$) are eigenvalues of $P_{l\cdot n_1}$ ($P_{m\cdot n_2}$) so that the
  proof follows.
  % The proof follows straight by Lemma~\ref{lmm:sym_path_eigvec_multiple}.
\end{proof}

We are now ready to characterize the eigenvector symmetries by suitable brick
partitions.
\begin{theorem}[Grid partition and eigenvector symmetries]
  \label{thm:subgrid_partition}
  Let $G_0 = P_{n_1} \Box P_{n_2}$ be a grid of dimension $n_1\times n_2$ with
  $P_{n_1}$ and $P_{n_2}$ paths of dimension respectively $n_1$ and $n_2$.
  Take any grid $G = P_{l\cdot n_1} \Box P_{m\cdot n_2}$ % be a grid
  of dimension $l n_1\times m n_2$
  % with $P_{l\cdot n_1}$ and $P_{m\cdot n_2}$ paths of dimension respectively
  % $l \cdot n_1$ and $m \cdot n_2$,
  and let $G_{ij}$, $i\in\until{l}$ and $j\in\until{m}$, be a partition into
  bricks of dimension $n_1 \times n_2$.
  % Also, let $G_0 = P_{n_1} \Box P_{n_2}$ a grid of dimension $n_1\times
  % n_2$. % Then the following holds
  % \begin{enumerate}
  % \item any eigenvalue $\lambda$ of the Laplacian $L_0$ of $G_0$ is an
  % eigenvalue of the Laplacian $L$ of $G$;
  % \item
		%
		
  %   Let $\lambda$ be a common (possibly non-simple) eigenvalue of $L$ and
  %   $L_0$
  %   and let $V_\lambda$ be the associated eigenspace of $L$. Then any
  %   eigenvector $v \in V_\lambda$ (i.e. any eigenvector of $L$ associated to
  %   $\lambda$) can be decomposed into subvectors $v_{ij}$ relative to the
  %   bricks $G_{ij}$ where
		
  Then for each eigenvalue (possibly non-simple) $\lambda$ of $L_0$:
\begin{enumerate}
\item $\lambda$ is an eigenvalue of $L$, and
\item any eigenvector $v$ of $L$ associated to $\lambda$ can be decomposed into sub-vectors $v_{ij}$
  relative to the bricks $G_{ij}$ with
  \[
  v_{ij} = (\Pi_{n_1} \otimes I_{n_2})^{(i-1)} (I_{n_1}\otimes
  \Pi_{n_2})^{(j-1)} v_0
  \]
  for $i\in\until{n_1}$ and $j\in\until{n_2}$, where $v_0$ is an eigenvector of
  $L_0$ associated to $\lambda$.
  \end{enumerate}
\end{theorem}

\begin{proof}
  Statement~(i) follows straight by Lemma~\ref{lmm:eigs_G0_G}.
  To prove statement~(ii), first, let us recall that the matrices $(\Pi_{n_1}
  \otimes I_{n_2})$ and $(I_{n_1}\otimes \Pi_{n_2})$ applied to the vectors
  $v_0$ respectively flip the first and second components. Also, $(\Pi_{n_1}
  \otimes I_{n_2})^{(i-1)} w = w$ for $i$ odd (respectively $(I_{n_1}\otimes
  \Pi_{n_2})^{(j-1)} w = w$ for $j$ odd). Thus, we can just prove the result for
  $(i,j) = (0,1)$ and $(i,j) = (1,0)$.

  Let $\mu$ be the geometric multiplicity of the eigenvalue $\lambda$ for the
  Laplacian $L$. Then, by Lemma~\ref{lmm:eigstruc_cartesian_prod}, a basis of
  the associated eigenspace $V_\lambda$ is given by $\mu$ vectors obtained as
  the Kronecker product of eigenvectors of the constitutive path graphs. That
  is,
  \[
  V_{\lambda} = \{ v \in \real^{n_1 \cdot n_2} | v = \sum_{i=1}^\mu \alpha_i
  (v_{i} \otimes w_{i}), \alpha_i \in \real \},
  \]
  where $v_i$ and $w_i$, $i\in\until{\mu}$ are eigenvectors of respectively
  $P_{n_1}$ and $P_{n_2}$ associated to eigenvalues $\lambda_{i,1}$ and
  $\lambda_{i,2}$ such that $\lambda_{i,1} + \lambda_{i,2} = \lambda$.
	%
  % Now, since each $\bar{v}_i$ and $\bar{w}_i$ belongs either to $S^+$ or to
  % $S^-$, without loss of generality, we can write any $v \in V_\lambda$ as
  %	\[
  % v = \alpha_1 (v_{1} \otimes w_{1}) + \alpha_2 (v_{2} \otimes w_{2}) +
  % \alpha_3 (v_{3} \otimes w_{3}) + \alpha_4 (v_{4} \otimes w_{4}),
  %	\]
  %	$\alpha_i \in \real$, $i\in\until{4}$, where $v_1$, $w_1$, $v_2$ and
  % $w_3$ belong to $S^+$, while $w_2$, $v_3$, $v_4$ and $w_4$ belong to
  % $S^-$. In other words, we can write $v$ as the linear combination of four
  % vectors being the Kronecker product of two vectors belonging respectively
  % to: $S^+$ and $S^+$, $S^+$ and $S^-$, $S^-$ and $S^+$, and $S^-$ and $S^-$.
  Exploiting the Kronecker product and using the result in
  Lemma~\ref{lmm:sym_path_eigvec_multiple}, we have
  % \[
  %	\begin{split}
  %   v = \sum_{i=1}^\mu \alpha_i
  %   \begin{bmatrix}
  %     (v_i)_1
  %     \begin{bmatrix}
  %       w_i\\
  %       \Pi w_i \vspace{-5pt}\\
  %				   % w_i\\
  % \vdots
  % \end{bmatrix}\\
  % \vdots\\
  % (v_i)_{n_1}
  % \begin{bmatrix}
  %   w_i\\
  %   \Pi w_i \vspace{-5pt}\\
  %				   % w_i\\
  % \vdots
  % \end{bmatrix}\\
  % \relbar\relbar\relbar\relbar\relbar\relbar\relbar\\
  % (\Pi v_i)_1
  % \begin{bmatrix}
  %   w_i\\
  %   \Pi w_i \vspace{-5pt}\\
  %				   % w_i\\
  % \vdots
  % \end{bmatrix}\\
  % \vdots\\
  % (\Pi v_i)_{n_1}
  % \begin{bmatrix}
  %   w_i\\
  %   \Pi w_i \vspace{-5pt}\\
  %				   % w_i\\
  % \vdots
  % \end{bmatrix}\\
  % \relbar\relbar\relbar\relbar\relbar\relbar\relbar\\
  % \vdots
  % \end{bmatrix}.
  % \end{split}
  %	\]

  \[
  \begin{split}
    v = \sum_{i=1}^\mu \alpha_i 
    % \begin{bmatrix}
    \left[ (v_i)_1
      \begin{bmatrix}
        w_i\\
        \Pi w_i\\
        w_i\\
        \vdots
      \end{bmatrix}^T \;\; \ldots \;\; (v_i)_{n_1}
      \begin{bmatrix}
        w_i\\
        \Pi w_i\\
        w_i\\
        \vdots
      \end{bmatrix}^T \;\;
      \vline \;\;\right.
     \left.  (\Pi v_i)_1
      \begin{bmatrix}
        w_i\\
        \Pi w_i\\
        w_i\\
        \vdots
      \end{bmatrix}^T \;\; \ldots \; (\Pi v_i)_{n_1}
      \begin{bmatrix}
        w_i\\
        \Pi w_i\\
        w_i\\
        \vdots
      \end{bmatrix}^T \; \vline \ldots \right]^T.
    % \end{bmatrix}.
  \end{split}
  \]
  Clearly, the brick $G_{11}$ coincides with the grid $G_0$. Thus, we can
  compare the bricks $G_{ij}$ with the brick $G_{11}$. The sub-eigenvector
  corresponding to the ``first row'' of the brick $G_{11}$ is given by
  \[
  (v_{11})_{[1, \,1:n_1]}:=
  \begin{bmatrix}
    (v_{11})_{[1,1]}\\
    \vdots\\
    (v_{11})_{[1,n_1]}\\
  \end{bmatrix}
  = \alpha_i (v_i)_1 w_i
  \]
  Using the definition of brick components, the sub-eigenvectors corresponding
  to the ``first row'' of the grid $G_{12}$ is
  \[
  \begin{split}
    (v_{12})_{[1, \,1:n_1]}:=
    \begin{bmatrix}
      (v_{12})_{[1,1]}\\
      \vdots\\
      (v_{12})_{[1,n_1]}\\
    \end{bmatrix}
    = \alpha_i (v_i)_1 \Pi w_i
    =
     \alpha_i (v_i)_1
    \begin{bmatrix}
      (w_i)_n\\
      \vdots\\
      (w_i)_1\\
    \end{bmatrix}
    =
    \begin{bmatrix}
      (v_{11})_{[1,n_1]}\\
      \vdots\\
      (v_{11})_{[1,1]}\\
    \end{bmatrix}.
  \end{split}
  \]
  The proof for the other components follows exactly the same arguments.
\end{proof}

The above theorem has a nice and intuitive graphical interpretation, as shown in
Figure~\ref{fig:grid_partition_S11}. 
\begin{figure}[htbp]
  \centering
  \includegraphics[height=.4\linewidth]{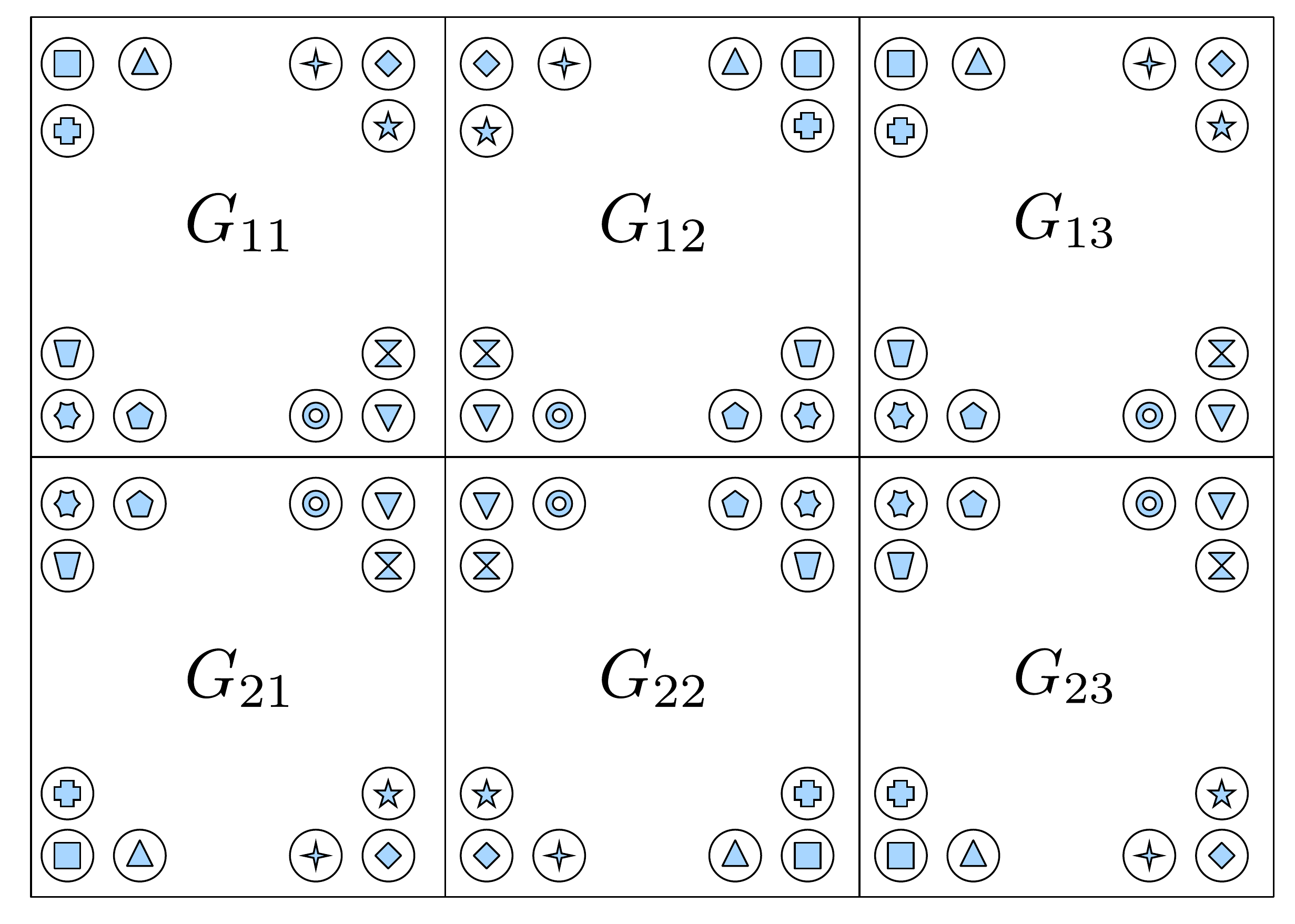}%
  % \vspace{-7 mm}
  % \vspace{-7 mm}
  \caption{Graphical interpretation of Theorem~\ref{thm:subgrid_partition}}
  \label{fig:grid_partition_S11}
\end{figure}

Given a grid $G$ and an eigenvector $v$
associated to an eigenvalue $\lambda$, we can associate a symbol to each node
depending on the value of the eigenvector component. Then, we partition the grid
$G$ into bricks of dimension $n_1 \times n_2$. Given the symbols in the
brick $G_{11}$, the symbols in a brick $G_{i,j}$, for $i\in\fromto{1}{l}$
and $j\in\fromto{2}{m}$, are obtained by a reflection of the brick $G_{i,j-1}$
with respect to the horizontal axis, while the symbols in a brick $G_{i,j}$,
for $i\in\fromto{2}{l}$ and $j\in\fromto{1}{m}$, are obtained by a reflection of
the brick $G_{i-1,j}$ with respect of the vertical axis.

Next, we analyze the eigenvector components of a brick whose dimensions are
prime or, equivalently, the components of eigenvectors associated to eigenvalues
that are not eigenvalues of smaller bricks.
Recalling that any path eigenvector $v$ satisfies either $v=\Pi v$ ($v\in S^+$)
or $v = -\Pi v$ ($v\in S^-$), we show how: (i) each basis eigenvector has a
symmetry induced by the symmetry of the path eigenvectors, and (ii) the
structure of a general grid eigenvector is influenced by the symmetry of the
basis eigenvectors.

\begin{proposition}
  \label{prop:subgrid_sym}
  Let $G_0 = P_{n_1} \Box P_{n_2}$ be a grid of dimension $n_1\times n_2$. For
  any eigenvalue $\lambda$, let $V_\lambda$ be the associated eigenspace, with
  structure as in equation \eqref{eq:V_lambda}. Then, each basis eigenvector
  generating $V_\lambda$, $(v_{i} \otimes w_{i})$, satisfies one of the four
  relations:
  \[
  \begin{split}
    (v_i \otimes w_i) &= \phantom{-}(\Pi \otimes I) (v_i \otimes w_i) = \phantom{-}(I \otimes \Pi) (v_i \otimes w_i)\\
    (v_i \otimes w_i) &= \phantom{-}(\Pi \otimes I) (v_i \otimes w_i) = -(I \otimes \Pi) (v_i \otimes w_i)\\
    (v_i \otimes w_i) &= -(\Pi \otimes I) (v_i \otimes w_i) = \phantom{-}(I \otimes \Pi) (v_i \otimes w_i)\\
    (v_i \otimes w_i) &= -(\Pi \otimes I) (v_i \otimes w_i)= -(I \otimes \Pi) (v_i \otimes w_i).\\
  \end{split}
  \]
\end{proposition}

\begin{proof}
  Using the result in Lemma~\ref{lem:sym_path_eigvec}, $v_i$ and $w_i$ belong
  either to $S^+$ or $S^-$, that is, e.g., either $v_i = \Pi v_i$ or $v_i = -\Pi
  v_i$. Suppose that, for example, $v_i\in S^+$ and $w_i\in S^-$. Under this
  assumption $(v_i \otimes w_i) = (\Pi v_i \otimes w_i) = (\Pi v_i \otimes I
  w_i)$ and, using the distributive property of the Kronecker product $(v_i
  \otimes w_i) = (\Pi \otimes I)(v_i \otimes w_i)$. Also, $(v_i \otimes w_i) =
  (v_i \otimes -\Pi w_i) = (I v_i \otimes -\Pi w_i) = -(I\otimes \Pi)(v_i
  \otimes w_i)$. This gives the second of the four relations. The other three
  cases follow by the other three possible combinations of the symmetries of
  $v_i$ and $w_i$, thus concluding the proof.
\end{proof}

In the following we denote the set of vectors satisfying each one of the four
relations in the proposition respectively as $S^{++}$, $S^{+-}$, $S^{-+}$ and
$S^{--}$. A graphical representation of the four sets is given in
Figure~\ref{fig:single_subgrid}.
\begin{figure}[htbp]
  \centering \subfloat[\small symmetry class
  $S^{++}$]{\includegraphics[height=.2\linewidth]{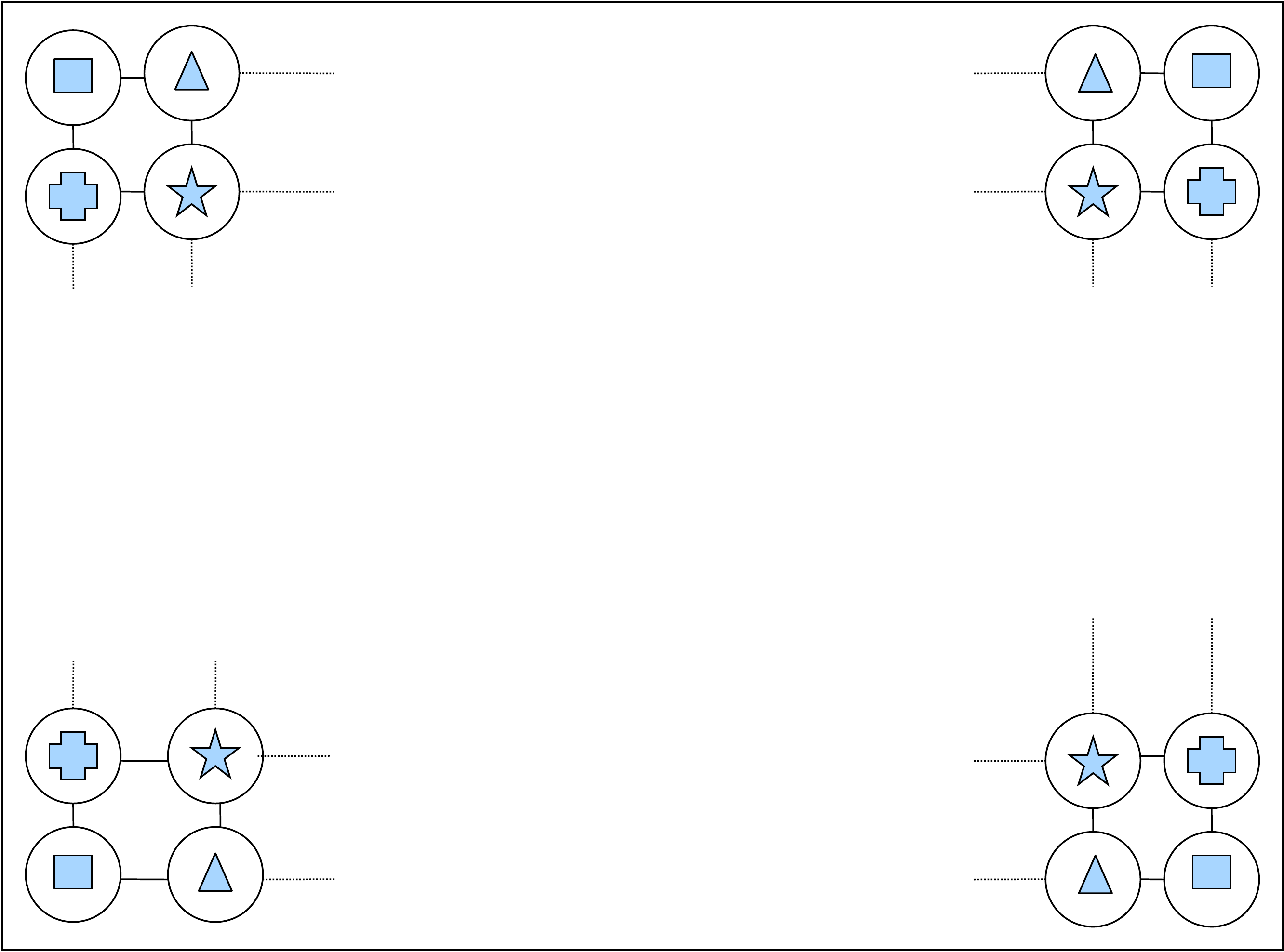}} \; %
  \subfloat[\small symmetry class
  $S^{+-}$]{\includegraphics[height=.2\linewidth]{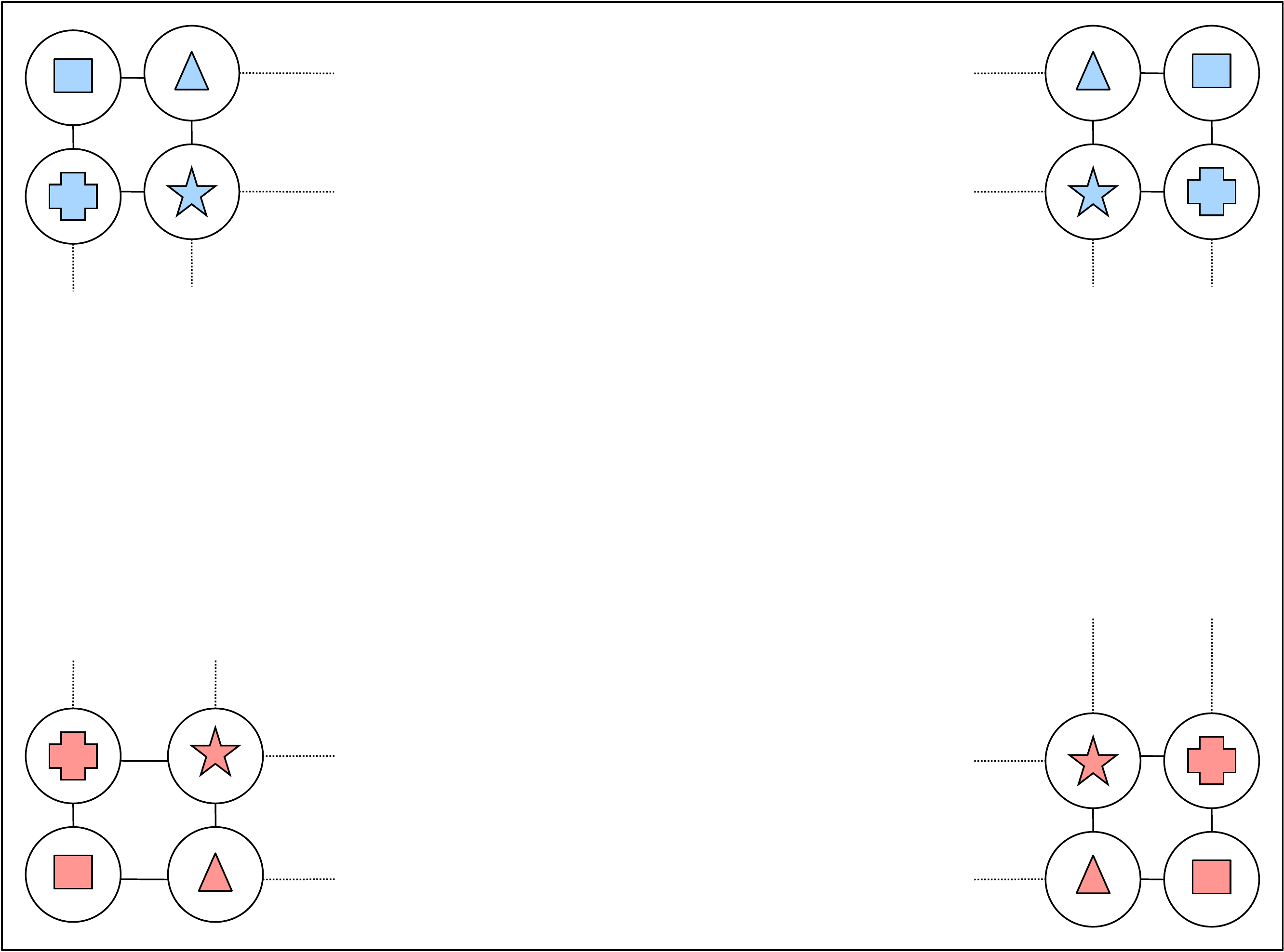}}\\[1.2ex]%
  \subfloat[\small symmetry class
  $S^{-+}$]{\includegraphics[height=.2\linewidth]{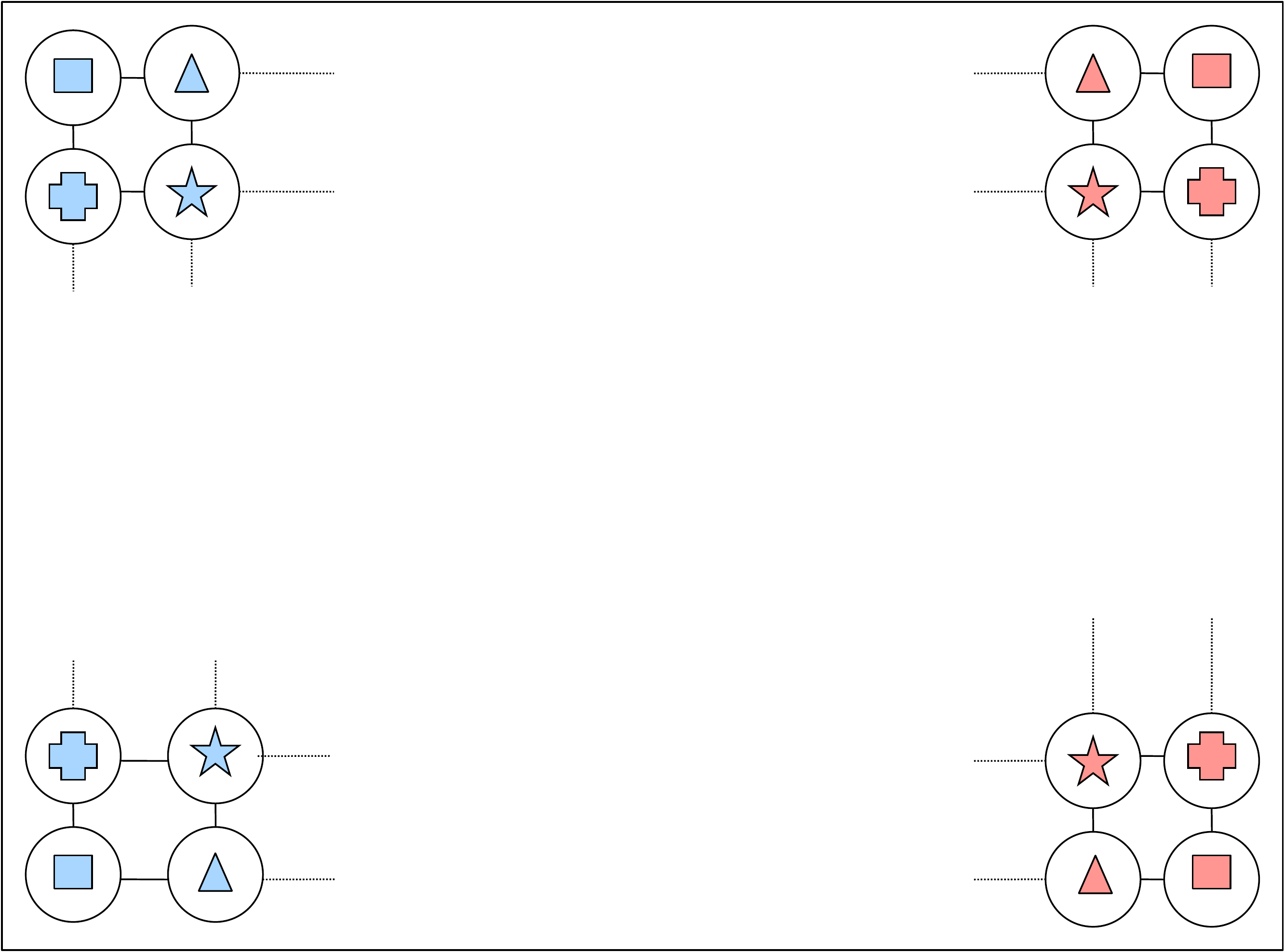}}\; %
  \subfloat[\small symmetry class
  $S^{--}$]{\includegraphics[height=.2\linewidth]{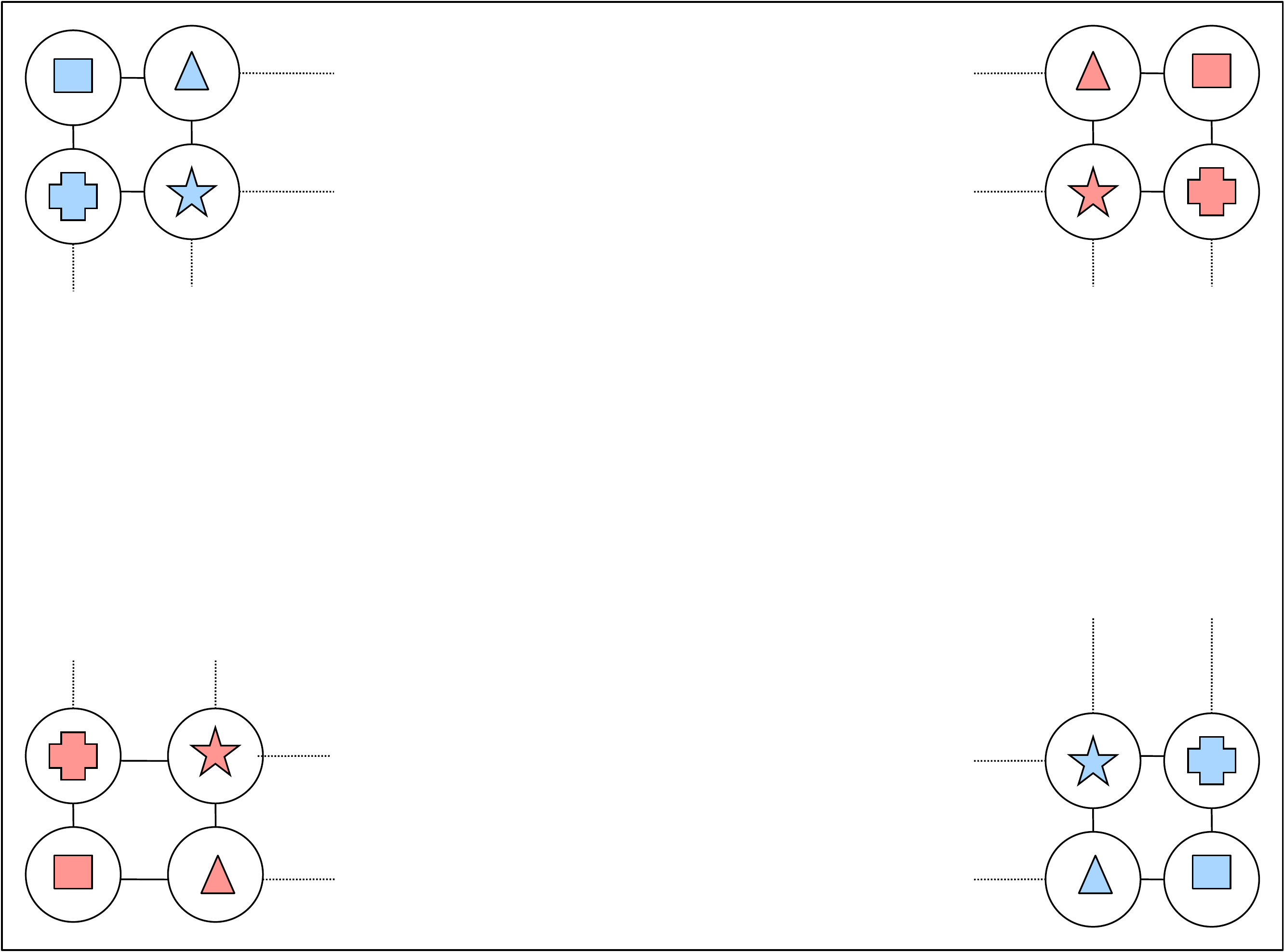}}%
  % \vspace{-7 mm}
  % \vspace{-7 mm}
  \caption{Graphical interpretation of Proposition~\ref{prop:subgrid_sym}}
  \label{fig:single_subgrid}
\end{figure}
We associate a symbol to each node depending on the value of the eigenvector
component. Also, we denote with the same symbol but different colors nodes that
have components of opposite sign. The result in
Proposition~\ref{prop:subgrid_sym} can be easily explained by using this
graphical interpretation. Namely, each of the four cases in the proposition
correspond to a scheme in Figure~\ref{fig:single_subgrid}.

The next remark gives an insight on the eigenvector components of the
``central'' nodes of a grid with odd dimensions.
\begin{remark}[Symmetries for grids with prime dimensions]
  If the grid has odd dimensions, $n_1$ and $n_2$, the above symmetries have
  interesting implications for the nodes with components respectively
  $[\frac{n_1+1}{2},\ell]$, $\ell\in\until{n_1}$, and $[\nu,\frac{n_2+1}{2}]$,
  $\nu\in\until{n_2}$. Indeed, the first set of components is zero for $S^{-+}$
  and $S^{--}$, while the second one is zero for $S^{+-}$ and $S^{--}$. \oprocend
\end{remark}

This proposition has an important impact on the symmetries of general
eigenvectors belonging to the same eigenspace (and in particular for each brick
of a general grid).
% which can be easily shown by means of its graphical interpretation.
%
Clearly, if an eigenvalue is simple, then any associated eigenvector has the
structure of a Kronecker product and thus one of the four symmetries. For a
non-simple eigenvalue, any eigenvector of $V_\lambda$ can be written as a linear
combination of the basis vectors, and thus, using the proposition, by the sum of
at most four vectors each one having one of the four symmetries. Thus, in order
to identify the symmetries of a general vector, we need to ``suitably combine''
nodes with the same symbol and color in different classes.
Three cases are possible: (i) if basis vectors of at least three different
classes are present, by inspection in Figure~\ref{fig:single_subgrid}, no
symmetries are present, (ii) if all basis vectors belong to the same class, then
also the linear combination does, and (iii) if the basis vectors belong to two
of the four classes, a general eigenvector $v$ (obtained as their linear
combination) satisfies the following symmetries:
\begin{itemize}
\item[a)] $(v)_{[\nu, \ell]}=\pm (v)_{[n_1-\nu+1, \ell]}$ if the two classes share the first
  symbol (e.g., $S^{++}$ and $S^{+-}$);
\item[b)] $(v)_{[\nu, \ell]}=\pm (v)_{[\nu, n_2-\ell+1]}$ if the two classes share the
  second symbol (e.g., $S^{++}$ and $S^{-+}$);
\item[c)] $(v)_{[\nu, \ell]}=\pm (v)_{[n_1-\nu+1, n_2-\ell+1]}$ if the two
  classes share no symbol (e.g., $S^{++}$ and $S^{--}$).
\end{itemize}
A graphical representation of the above three symmetries is depicted in
Figure~\ref{fig:general_eigvecs_sym}. We associate the same symbol to nodes
having the same absolute value of the eigenvector component. It is worth noting
that we are interested in the absolute values because we want to classify all
the components that can be zeroed simultaneously.

\begin{figure}[htbp]
  \centering \subfloat[\small first symbol in common]{\includegraphics[height=.2\linewidth]{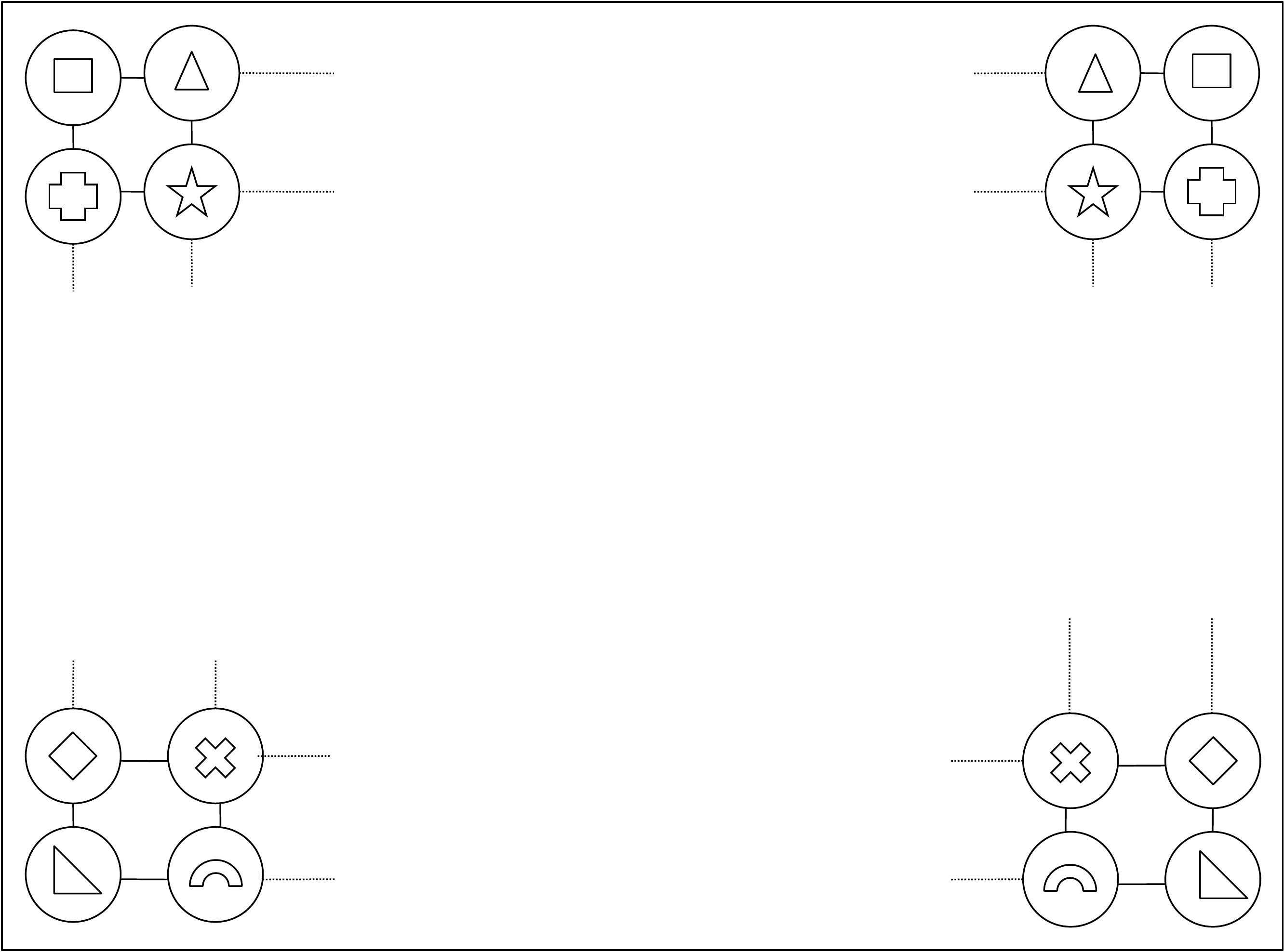}} \; %
  \subfloat[\small second symbol in common]{\includegraphics[height=.2\linewidth]{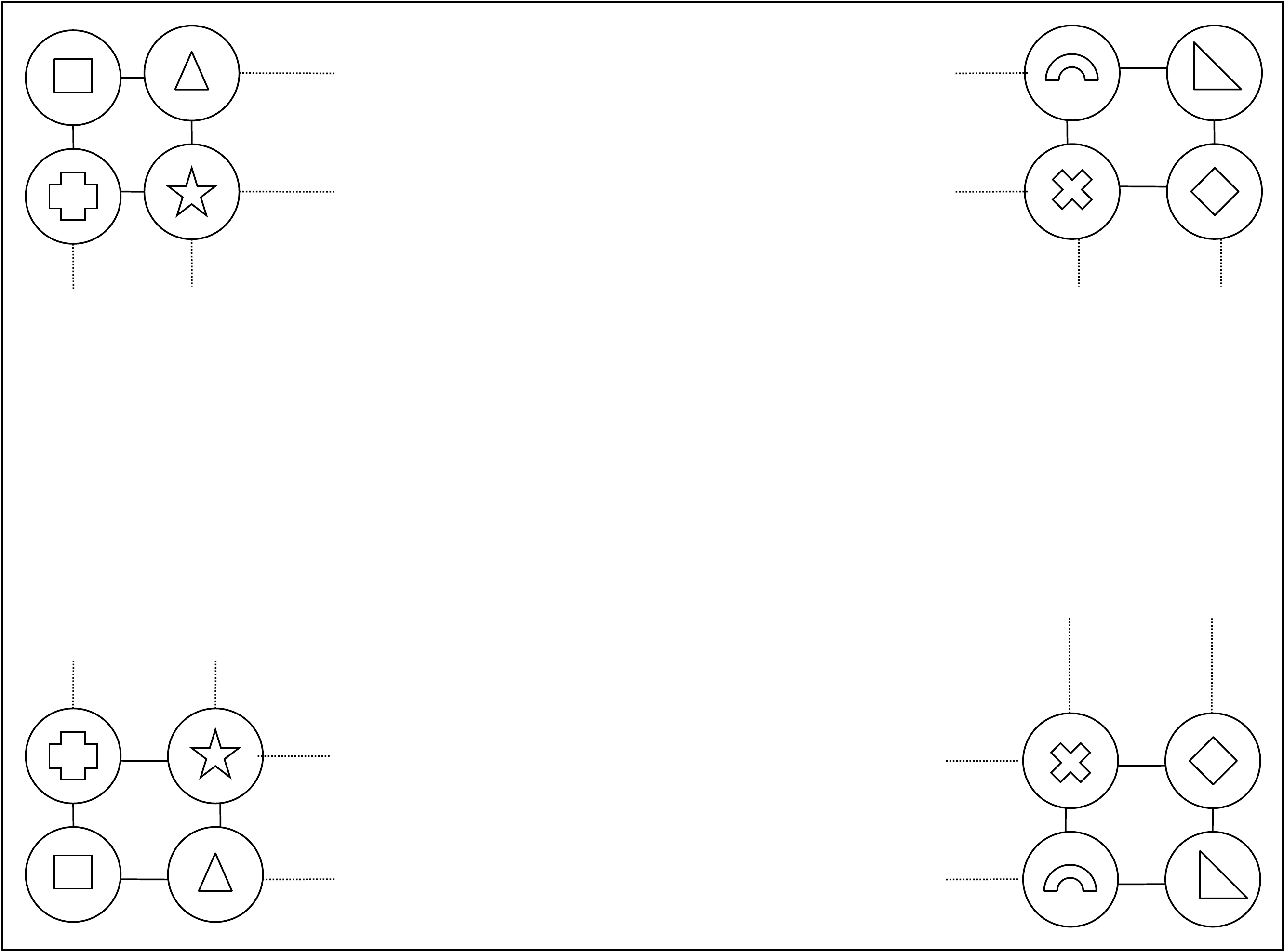}}\; %
  \subfloat[\small no symbol in common]{\includegraphics[height=.2\linewidth]{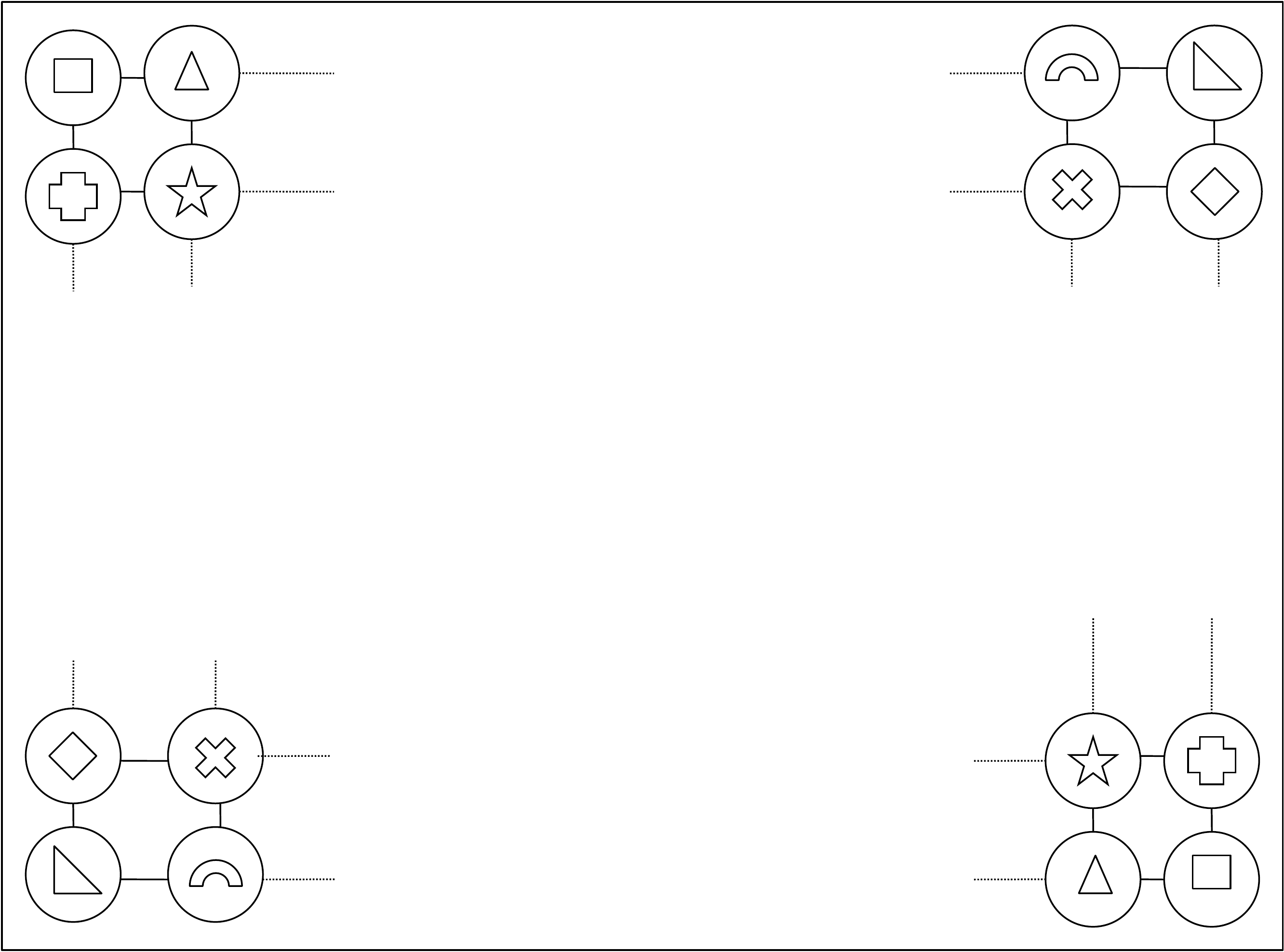}}\; %
  \caption{Symmetries of eigenvectors obtained as the sum of
  basis vectors of two different classes.}
  \label{fig:general_eigvecs_sym}
\end{figure}

\section{Controllability and observability analysis of general grid graphs}
\label{sec:general_grids_observ}
In this section we provide necessary and sufficient conditions to characterize
all and only the nodes from which the network system is controllable (observable).
First, we need a well known result in linear systems theory, see, e.g.,
\cite{PJA-ANM:94}. We state it for the controllability property.
\begin{lemma}
  \label{lmm:unobs_multi_eigs}
  If a state matrix $A\in\real^{n\times n}$, $n\in\natural$, has an eigenvalue
  with geometric multiplicity $\mu \in\natural$, then for any
  $B\in\real^{n\times m}$ with $m<\mu$ the pair $(A,B)$ is uncontrollable.\oprocend
\end{lemma}
The previous lemma applied to the grid Laplacian says that, in case the grid is
non simple with maximum eigenvalue multiplicity $\mu$, then the grid is not
controllable (observable) from a set of control (observation) nodes of cardinality less than $\mu$.

Using Lemma~\ref{lmm:PBH_eigvec}, it follows straight that we can study the
controllability (observability) properties of the grid separately for each eigenvalue. Namely, to
guarantee controllability (observability), we need to show that for each eigenvalue of the grid
Laplacian $L$, there does not exist any eigenvector satisfying the condition in
\eqref{eq:PBH_eigvec_reach_obs}, i.e. having zero in some components.

% More formally, let $\lambda = \lambda_{i} + \lambda_{j}$ be a simple
% eigenvalue of the grid corresponding to the path eigenvalues $\lambda_i$ and
% $\lambda_j$, and $v_{ij} = v_i \otimes v_j$ the corresponding eigenvector
% (with $v_i$ and $v_j$ the corresponding eigenvectors of the paths).

If $\lambda$ is simple, the corresponding eigenspace $V_\lambda$ in
\eqref{eq:V_lambda} is given by $V_\lambda = \{ v \in \real^{n_1 \cdot n_2} | v
= \alpha_1 (v_{1} \otimes w_{1}), \alpha_1 \in \real \}.$ Thus, finding the
zeros of any eigenvector in $V_\lambda$ is equivalent to finding the zeros of
the eigenvectors $v_1$ and $w_1$ and propagate them according to the Kronecker
product structure.
Clearly, with this observation in hand, the analysis of any simple eigenvalue
can be performed by using the tools provided in
Section~\ref{sec:simple_grids}.

For eigenvalues with multiplicity greater than one, next two considerations are
important. First, not all the eigenvectors of $\lambda$ have the structure of a
Kronecker product. Second, consistently with Lemma~\ref{lmm:unobs_multi_eigs},
it is always possible to find an eigenvector $v\in V_{\lambda}$ with an
arbitrary component equal to zero, for a suitable choice of the coefficients
$\alpha_i$ in
\eqref{eq:V_lambda}. %
% (e.g. for multiplicity two $\alpha_i$ such that $\alpha_1 (v_{1} \otimes
% w_{1})_\ell + \alpha_2 (v_2 \otimes
% w_2)_\ell = 0$ to set to zero the $\ell$th component).
Thus, the controllability (observability) analysis does not depend only on the zero components of
the path eigenvectors, but also on the symmetries in the grid eigenvector
components. That is, for the eigenvalue under investigation, we want to answer
to the following question. If we find an eigenvector with zero in an arbitrary
component $\ell$, what are all and only the other components that are zero in
the chosen eigenvector? We provide the analysis for non-simple eigenvalues of
multiplicity two and discuss the generalization in a remark.

On the basis of the eigenvector symmetries identified in
Theorem~\ref{thm:subgrid_partition}, we can study the controllability
(observability) of a brick. 
% In particular, we concentrate our analysis on the
% sector of components $[\nu,\ell]$ with $\nu\in\until{\frac{n_1-1}{2}}$ and
% $\ell\in\until{\frac{n_2-1}{2}}$.

Next lemma provides useful properties of the eigenvector components in a brick.
% with prime length dimensions.
%
\begin{lemma}[Polynomial structure of the eigenvector components]
  \label{lmm:poly_components}
  Let $G_0 = P_{n_1} \Box P_{n_2}$ be a grid of dimension $n_1\times n_2$. Then,
  any Laplacian eigenvector $u = v \otimes w$ of the grid, with $v$ and $w$
  respectively eigenvectors of $P_{n_1}$ and $P_{n_2}$ associated to eigenvalues
  $\lambda_v$ and $\lambda_w$, has components $(u)_{[\nu,\ell]}$,
  $\nu\in\until{n_1}$ and $\ell\in\until{n_2}$ satisfying
  \begin{enumerate}
  \item $(u)_{[\nu,\ell]} = p_\nu(\lambda_v) \cdot p_{\ell}(\lambda_w) \cdot (v)_1
    \cdot (w)_1$, where $p_r(s)$ is the polynomial of degree $(r-1)$ defined as
    $p_2(s) = 1-s$ for $r=2$ and, denoting $p_1(s) = 1$, by the recursion
    \begin{equation}
      \label{eq:poly}
      p_r(s) = (2-s) p_{r-1}(s) - p_{r-2}(s)
    \end{equation}
    for $r\geq 3$;
  \item if $n_1$ and $n_2$ prime, then $p_\nu(\lambda_v) \neq 0$ and
    $p_{\ell}(\lambda_w) \neq 0$ for any $\nu\in\until{\frac{n_1-1}{2}}$ and
    $\ell\in\until{\frac{n_2-1}{2}}$.
  \end{enumerate}
\end{lemma}

\begin{proof}
  First, notice that $(u)_{[\nu,\ell]} = (v)_\nu \cdot (w)_\ell$. To prove
  statement~(i), we need to prove that for a path of length $n_1$, any
  eigenvector $v$ satisfies $(v)_\nu = p_\nu(\lambda_v) (v)_1$ for $\nu\in
  \until{n_1}$. We prove the statement by induction. We exploit the eigenvector
  relation $L_{n_1} v = \lambda_v v$ by using the structure of the path
  Laplacian given in Appendix. From the first row, it follows that $(v)_1 -
  (v)_2 = \lambda_v (v)_1$, so that $(v)_2 = (1-\lambda_v) (v)_1$.  Then, from
  the $r$th row of the relation, we have $-(v)_{r-1} + 2 (v)_r - (v)_{r+1} =
  \lambda_v (v)_r$ and thus $(v)_{r+1} = (2-\lambda_v) (v)_r -
  (v)_{r-1}$. Plugging in the inductive assumption $(v)_r = p_r(\lambda_v)
  (v)_1$ and $(v)_{r-1} = p_{r-1}(\lambda_v) (v)_1$, we have $(v)_{r+1} =
  p_{r+1}(\lambda_v) (v)_1$ with $p_{r+1}(\lambda_v) = (2-\lambda_v)
  p_{r}(\lambda_v) - p_{r-1}(\lambda_v)$ which concludes the first part of the
  proof.
	
  Statement~(ii) can be proven by showing that, for a path graph of length $n_1$
  with $n_1$ prime, any eigenvector $v$ has non zero components $(v)_1$,
  $\ldots$, $(v)_{\frac{n_1-1}{2}}$. This result is proven in \cite{GP-GN:10b},
  thus concluding the proof.
\end{proof}

Next theorem gives necessary and sufficient conditions for two eigenvector
components to be both zero in a brick. % with prime dimensions.
\begin{theorem}[Simultaneous zeroing of eigenvector components]
  \label{thm:poly_ratios}
  Let $G_0 = P_{n_1} \Box P_{n_2}$ be a grid of dimension $n_1\times n_2$. Let
  $\lambda = \lambda_{1,1} + \lambda_{1,2} = \lambda_{2,1} + \lambda_{2,2}$ be
  an eigenvalue of multiplicity two, with $\lambda_{1,1}$ and $\lambda_{2,1}$
  ($\lambda_{1,2}$ and $\lambda_{2,2}$) eigenvalues of $P_{n_1}$
  ($P_{n_2}$). Let $V_{\lambda}$ be the associated eigenspace. Then there exists
  an eigenvector $v\in V_\lambda$ with zero components $[\nu_1, \ell_1]$ and
  $[\nu_2, \ell_2]$, $\nu_1,\nu_2\in\until{n_1}$ and
  $\ell_1,\ell_2\in\until{n_2}$, if and only if
  \begin{equation}
    p_{\nu_1}(\lambda_{1,1}) \cdot p_{\ell_1}(\lambda_{1,2}) \cdot p_{\nu_2}(\lambda_{2,1}) \cdot
    p_{\ell_2}(\lambda_{2,2}) = \; p_{\nu_1}(\lambda_{2,1}) \cdot
    p_{\ell_1}(\lambda_{2,2}) \cdot
    p_{\nu_2}(\lambda_{1,1}) \cdot \; p_{\ell_2}(\lambda_{1,2}), 
    \label{eq:poly_determ}
  \end{equation}
  where $p_r(s)$ is the polynomial of degree $r-1$ defined by the recursion in
  equation \eqref{eq:poly}.
\end{theorem}

\begin{proof}
  To prove the statement, we look for an eigenvector $v = \alpha_1 (v_1 \otimes
  w_1) + \alpha_2 (v_2 \otimes w_2)$ with $\alpha_1$ and $\alpha_2$ such that
  $(v)_{[\nu_1,\ell_1]} = 0$ and $(v)_{[\nu_2,\ell_2]} = 0$. The condition
  $(v)_{[\nu_1,\ell_1]} = 0$ is equivalent to $\alpha_1 (v_1)_{\nu_1}
  (w_1)_{\ell_1} + \alpha_2 (v_2)_{\nu_1} (w_2)_{\ell_1} = 0$. From
  Lemma~\ref{lmm:poly_components} (i) we have $\alpha_1\,
  p_{\nu_1}(\lambda_{1,1})\cdot (v_1)_{1} \cdot p_{\ell_1}(\lambda_{2,1}) \cdot
  (w_1)_{1} + \alpha_2\, p_{\nu_1}(\lambda_{1,2}) \cdot (v_2)_{1} \cdot
  p_{\ell_1}(\lambda_{2,2}) \cdot (w_2)_{1} = 0$. Using the same calculations
  for the condition $(v)_{[\nu_2,\ell_2]} = 0$, we can write the matrix equation
  \begin{equation}
    \begin{bmatrix}
      p_{\nu_1}(\lambda_{1,1}) p_{\ell_1}(\lambda_{1,2}) (v_1)_1 (w_1)_1 \; & \;
      p_{\nu_1}(\lambda_{2,1}) p_{\ell_1}(\lambda_{2,2}) (v_2)_1
      (w_2)_1\\[1.2ex]
      p_{\nu_2}(\lambda_{1,1}) p_{\ell_2}(\lambda_{1,2}) (v_1)_1 (w_1)_1 \; & \;
      p_{\nu_2}(\lambda_{2,1}) p_{\ell_2}(\lambda_{2,2}) (v_2)_1 (w_2)_1
     \end{bmatrix} 
     \begin{bmatrix}
       \alpha_1\\
       \alpha_2
     \end{bmatrix}
     = 0.
     \label{eq:poly_matrix}
   \end{equation}
   Since $v$ is an eigenvector, $\alpha_1$ and $\alpha_2$ can not be zero
  simultaneously. Thus, the above equation is satisfied if and only if the
  matrix is singular. Imposing the condition that the determinant be zero gives
  equation~\eqref{eq:poly_determ}.
\end{proof}

\begin{remark}[Extensions to higher grid dimension and eigenvalue multiplicity]
  If the grid is of dimension $d >2$ the above theorem extends in a
  straightforward manner. Indeed, the result in Lemma~\ref{lmm:poly_components}
  (i) can be easily generalized as $(u)_{[i_1, \ldots, i_d]} =
  p_{i_1}(\lambda_{i_1}) \cdot \ldots \cdot p_{i_d}(\lambda_{i_d}) \cdot (v_1)_1
  \cdot \ldots \cdot (v_d)_1$, with suitable adaptation of the notation, so that
  the condition in equation~\eqref{eq:poly_determ} follows straight.

  If the grid has an eigenvalue of multiplicity $\mu > 2$, then the theorem
  generalizes by considering $\mu$ nodes. The condition in
  equation~\eqref{eq:poly_determ} follows by setting to zero the determinant of
  a $\mu\times \mu$ version of the matrix in \eqref{eq:poly_matrix} with
  elements $(i,j)\in\until{\mu}\times\until{\mu}$ given by $p_{\nu_i}(\lambda_{j,1})
  p_{\ell_i}(\lambda_{j,2}) (v_j)_1 (w_j)_1$. \oprocend
\end{remark}

Next, we show a graphical interpretation of the controllability (observability) results obtained
by combining the results of Theorem~\ref{thm:main_thm_simple_grid},
Theorem~\ref{thm:subgrid_partition}, Proposition~\ref{prop:subgrid_sym} and
Theorem~\ref{thm:poly_ratios}. We present it through an example. In
Figure~\ref{fig:4x6} we show a two dimensional grid of length $4\times 6$. 
The analysis for the simple eigenvalues can be performed as explained in
Section~\ref{sec:simple_grids}. This gives the cross symbol in the set of nodes
$[2,i]$ and $[5,i]$, $i\in\until{4}$ in Figure~\ref{fig:4x6}~(b).
Then, we partition the grid into bricks of dimensions $2\times 2$ and $2\times 3$. The
eigenvalue $\lambda_1 = 2$ (respectively $\lambda_2 = 3$) is an eigenvalue of
multiplicity two in the brick $2\times 2$ ($2\times 3$). The eigenvectors
generating $V_{\lambda_1}$ ($V_{\lambda_2}$) belong to $S^{+-}$ and $S^{-+}$
($S^{++}$ and $S^{--}$). This gives the symmetries in Figure~\ref{fig:4x6}~(a)
according to Proposition~\ref{prop:subgrid_sym} and the subsequent discussion.
Using Theorem~\ref{thm:poly_ratios} it is easy to verify that all different
symbols in Figure~\ref{fig:4x6}~(a) correspond, in fact, to distinct component values.
Replicating the brick symbols according to
Theorem~\ref{thm:subgrid_partition} we get the structure in
Figure~\ref{fig:4x6}~(b).
Notice that in this particular case we have used
the same cross symbol both for the non-simple eigenvalue $\lambda=3$ and for the
simple eigenvalues.
Finally, it can be easily tested that $\lambda_1 = 2$ and $\lambda_2 = 3$ are
the only two non-simple eigenvalues.
Given a set of control (observation) nodes, the grid is controllable (observable)
if and only if the nodes do not have any symbol in common. If, for example, the
control (observation) nodes share the top symbol, then the eigenvalue $\lambda = 2$ (of
the brick $2\times 2$) is uncontrollable (unobservable).
As for the simple case we let the reader play with the rule.

\begin{figure}[htbp]
  \centering \subfloat[bricks $2$x$2$ \&
  $2$x$3$]{\includegraphics[height=.3\linewidth]{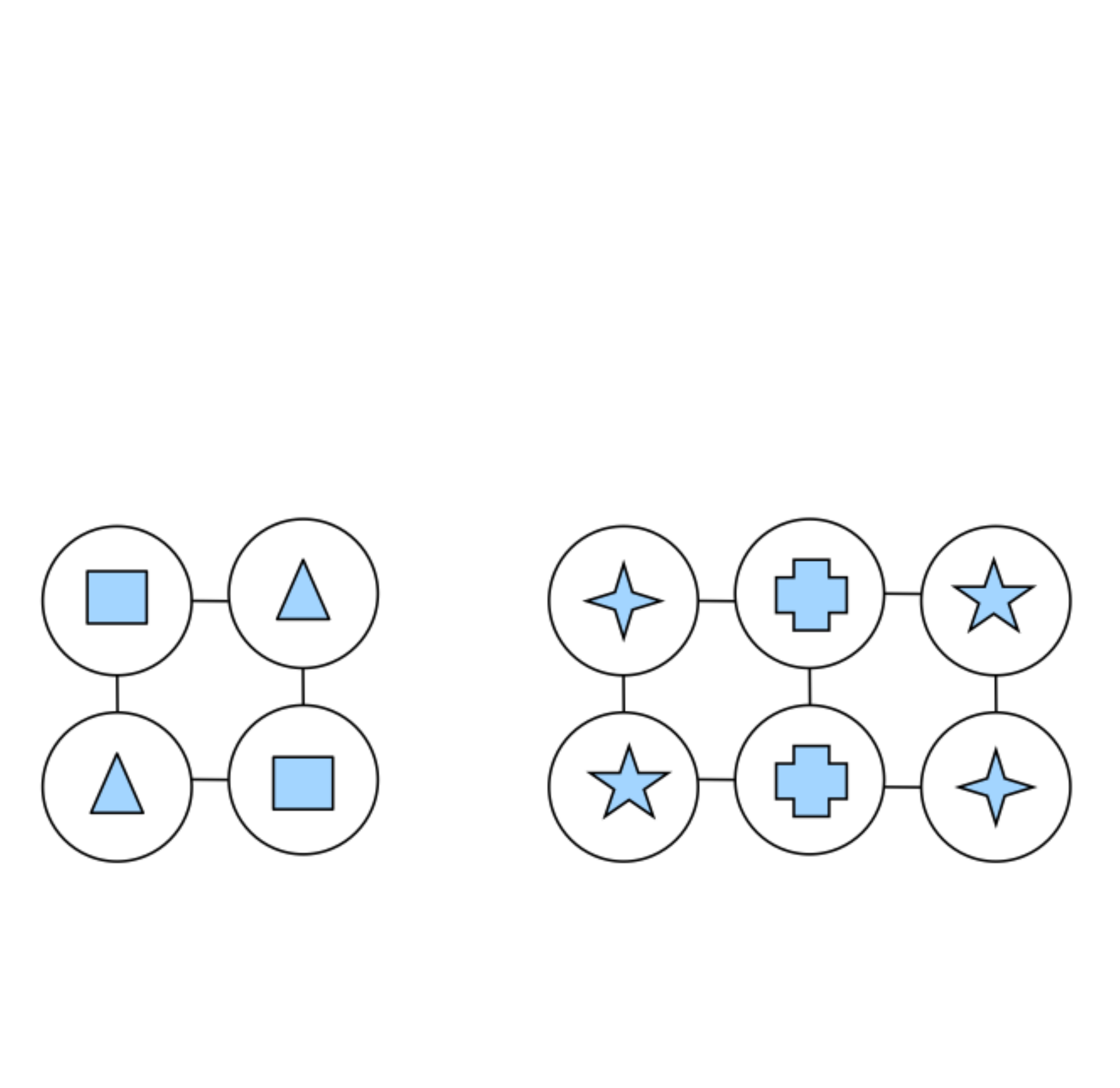}}
  \hspace{1cm}% \; %
  \subfloat[brick
  partition]{\includegraphics[height=.3\linewidth]{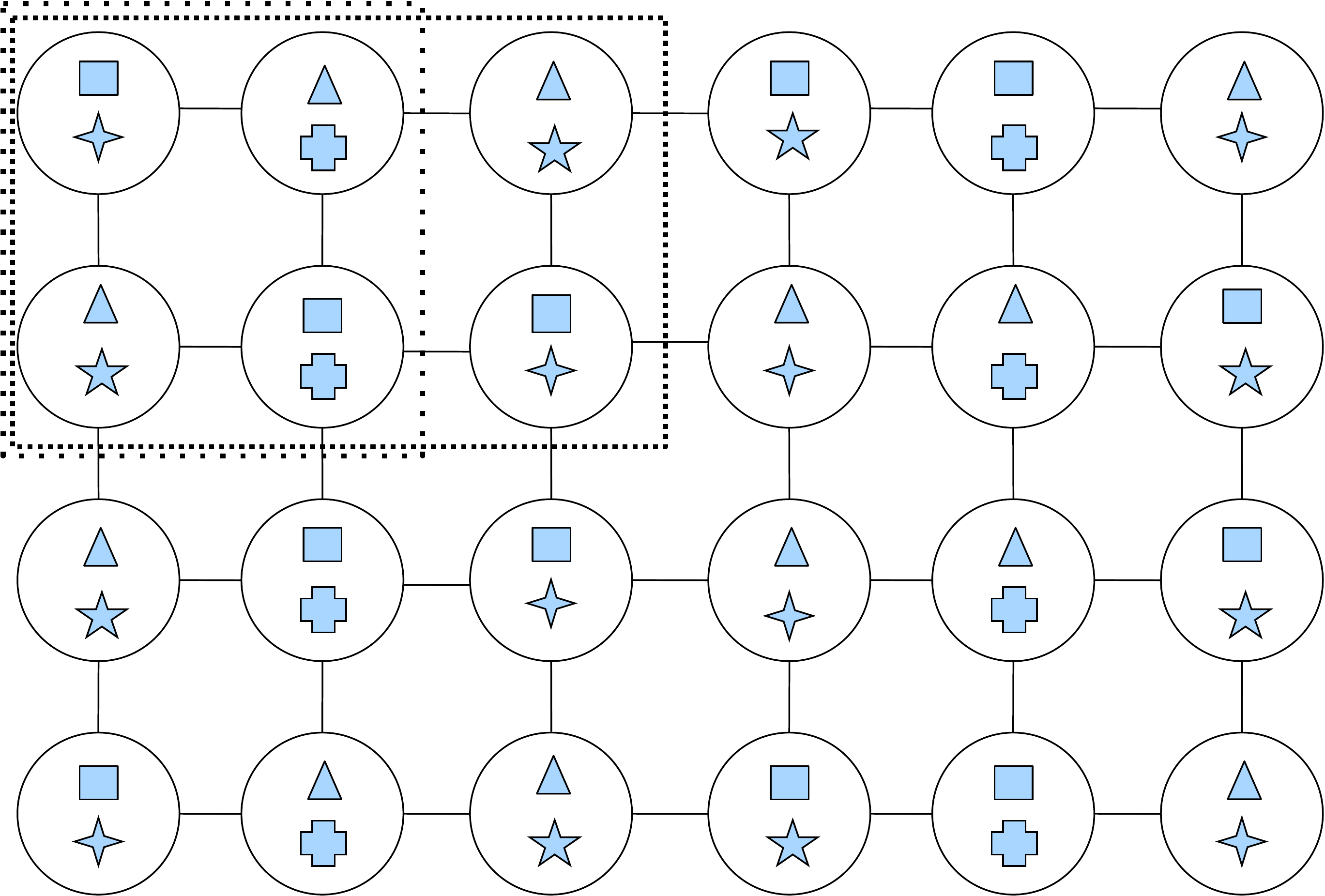}}%
  % \vspace{-7 mm}
  % \vspace{-7 mm}
  \caption{Graphical interpretation of the controllability and observability analysis.}
  \label{fig:4x6}
\end{figure}

To conclude, we provide a discussion on the importance and the effectiveness of
the proposed methodologies in studying the controllability and observability (in
general the dynamics) of grid graph induced systems. First, we want to stress
the fact that the proposed tools give strong insights on the structure and
symmetries of the grid eigenvectors and of the controllable (unobservable)
subspaces, as shown, e.g., in the above example. Furthermore, the proposed tools
represent, clearly, an effective alternative to the standard tests in checking
the controllability and observability properties. Notice that, the PBH test in
Lemma~\ref{lmm:PBH_eigvec} (to be performed for each eigenvalue) becomes
prohibitive as the dimensions of the grid grow.  Similarly, inspecting the rank
of the controllability (observability) matrix is an operation that is
ill-conditioned as the matrix dimension grows. As opposed to it, our tools
involve the following operations. The test on the simple eigenvalues can be done
simultaneously by using the tools in Section~\ref{sec:simple_grids} and involves
only arithmetic operations from number theory. The analysis for non-simple
eigenvalues involves the following operations. First, using
Theorem~\ref{thm:subgrid_partition}, the grid can be partitioned into bricks of
prime dimensions.  This operation is based on a straightforward prime number
factorization. Second, one has to compute the non simple eigenvalues for each
brick. This can be done by using Lemma~\ref{lmm:eigstruc_cartesian_prod} and the
closed form expression for the path eigenvalues given in Appendix. Third, for
each multiple eigenvalue, one has to inspect the symmetries in each brick (again
with simple operations on the brick dimensions according to
Proposition~\ref{prop:subgrid_sym}) and the possible coincidence of symbols (by
means of polynomial evaluations from Theorem~\ref{thm:poly_ratios}). Finally,
one should verify if there are multiple eigenvalues of the main grid that
are not eigenvalues of smaller bricks. However, so far we have never found such
a case in simulations, so that we conjecture that this is unlike or even
impossible to happen.

% \clearpage
\section{Conclusions}
\label{sec:conclusions}
In this paper we have characterized the controllability (by duality the
observability) of linear time-invariant systems whose dynamics are induced by
the Laplacian of a grid (or lattice) graph. We have shown that these systems
arise in several fields of application as, in particular, distributed control
and estimation, quantum computation and approximate solution of partial
differential equations. We have characterized the eigenstructure of the grid
Laplacian in terms of suitable graph decompositions and symmetries, and in terms
of simple rules from number theory. Based on this analysis, we have shown what
are all and only the uncontrollable (unobservable) set of nodes and provided
simple routines to choose a set of control (observation) nodes that guarantee
controllability (observability).

% In this paper we have characterized the controllability (by duality the
% observability) of grid (or lattice) graphs in terms of suitable graph
% decompositions, symmetries in the structure of the grid eigenvectors and simple
% rules from number theory. In particular, we have shown what are all and only the
% uncontrollable (unobservable) set of nodes and provided simple routines to choose a set of
% control (observation) nodes that guarantee controllability (observability).

\appendix

\section{ Previous results on the controllability and observability of path and cycle graphs}
\label{sec:previous_results}
In this section we briefly recall the results in \cite{GP-GN:12}, see also
\cite{GP-GN:10b}, on the controllability (observability) of path graphs. The characterization of
the controllability (observability) for grid graphs relies on these results.

Since it is extensively used in the paper, we provide the %explicit
expression of the path Laplacian, $L_n$, and of its distinct eigenvalues $\lambda_1,
\ldots, \lambda_n$
\begin{equation*}
  L_n =
  \begin{bmatrix}
    \;1      &      \text{-}1 & 0        &    \dots       & \;0 \\
    \text{-}1      &       \;2 &   \text{-}1      &   \dots     &  \;0\\
    \vdots &         & \ddots      &            &     \\
    \;0      &         &               \text{-}1    &   \;2 &      \text{-}1\\
    \;0      &         &               \;0    &   \text{-}1 &       \;1\\
  \end{bmatrix}
\quad
\text{and}
\quad
 \lambda_{k} = 2-2\cos\left((k-1) \frac{\pi}{n}\right), k=1,\ldots,n.
  % \label{eq:path_laplacian}
\end{equation*}

% \subsection{Observability of path graphs}
The controllability (observability) of the path can be analyzed by using the PBH lemma in the form
expressed in Lemma~\ref{lmm:PBH_eigvec}.
First, it is known, \cite{GP-GN:10a}, that a path graph is always controllable (observable)
from an external node ($1$ or $n$).
Next theorem, which is Theorem~4.4 in \cite{GP-GN:10b}, completely characterizes
the controllability (observability) of a path by means of simple rules from number theory.

\begin{theorem}[Path controllability and observability]
  \label{thm:main_thm_path}
  Given a path graph of length $n$, let $n = 2^{n_0} \prod_{\nu =1}^k p_\nu$ be
  a prime number factorization for some $k\in \natural$ and distinct (odd) prime
  numbers $p_1, \ldots, p_{k}$. The following statements hold:
  \begin{enumerate}
    % \item The unobservable nodes are all and only the nodes $i\in\until{n}$
    %   such that
    %   \[
    %   \text{mod}((n-i) - (i-1))_{p} = 0,
    %   \]
    %   for some (odd prime number) $p$ dividing n;
    %   \label{thm:main_thm_path_i}
  \item the path is not completely controllable (observable) from a node
    $i\in\fromto{2}{n-1}$ if and only if
    \[
    (n-i)\; \eqmod{p} \;(i-1)
    \]
    for some odd prime $p$ dividing $n$;
  \item the path is not completely controllable (observable) from a set of nodes
    $I_s = \{i_1, \ldots, i_m\} \subset \fromto{2}{n-1}$ if and only if
    \[
    \begin{split}
      2 (i_1-1) + 1 \eqmod{p} (i_2-i_1) \eqmod{p} %\ldots\\
      \ldots \eqmod{p} i_{m}-i_{m-1} & \eqmod{p} 2 (n-i_m) + 1,
    \end{split}
    \]
    for some odd prime $p$ dividing $n$;
  \item \label{pt:set_unobs_p} for each odd prime factor $p\in\{p_1,\ldots,
    p_k\}$ of $n$,
    % all nodes with index $i p_j - \frac{p_j-1}{2}$,
    % $i\in\until{\frac{n}{p_j}}$ are uncontrollable (unobservable) and share the
    % following uncontrollable (unobservable) eigenvalues
    the path is not controllable (observable) from each set of nodes $I_s^p = \{
    \ell p - \frac{p-1}{2}\}_{\ell\in\until{\frac{n}{p}}}$
    % all nodes with index $\ell p - \frac{p-1}{2}$,
    % $\ell\in\until{\frac{n}{p}}$,
    with the following uncontrollable (unobservable) eigenvalues
    \begin{equation}
      \begin{split}
        \lambda_{\nu} = 2-2\cos\!\left(\!\! (2 \nu-1) \frac{\pi}{{p}}\!\right)\!,\;
        \nu \in\until{\frac{p-1}{2}};
      \end{split}
      \label{eq:unobs_eigs_path}
    \end{equation}
    and uncontrollable (unobservable) eigenvectors
    \begin{equation}
      \begin{split}
        V_{\nu} =
        \begin{bmatrix}
          v_{\nu}^T & \!\!0 &\!\!-(\Pi v_{\nu})^T &\!\!-v_{\nu}^T & \!\!0
          % \ldots\\[1ex]
          % &
          % \ldots &0 &(-1)^{\frac{n}{p_j}-1}(\Pi v_{\nu,j})^T
          % &(-1)^{\frac{n}{p_j}-1}v_{\nu,j}^T &0
          & \!\!\!\ldots &\!\!\! (-1)^{\frac{n}{p}}(\Pi v_{\nu})^T
        \end{bmatrix}^T
      \end{split}
      \label{eq:unobs_eigvecs_path}
    \end{equation}
    % \[
    % \begin{bmatrix}
    %   v_{\nu,j}\\
    %   0\\
    %   \Pi v_{\nu,j}\\
    %   v_{\nu,j}\\
    %   \vdots
    %   0\\
    %   \Pi v_{\nu,j}\\
    %   v_{\nu,j}\\
    %   0\\
    %   \Pi v_{\nu,j}
    % \end{bmatrix},
    % \]
    where $v_{\nu}\in\real^{(p-1)/2}$ is the eigenvector of $\Lu_{(p-1)/2}$
    corresponding to the eigenvalue $\lambda_{\nu}$ for $\nu
    \in\until{(p-1)/2}$; and
  \item if node $i$ belongs to $I_s^{q_j}=\{ \ell q_j -
    \frac{q_j-1}{2}\}_{\ell\in\until{\frac{n}{q_j}}}$ for $l\leq k$ distinct
    prime factors $q_{1}\neq \ldots \neq q_{l}$ of $n$, then the set of
    uncontrollable (unobservable) eigenvalues from node $i$ is given by
    \begin{equation*}
      \begin{split}
        \lambda_{\nu} = 2-2&\cos\left((2 \nu-1) \frac{\pi}{{q_1\cdot \ldots
              \cdot
              q_l}}\right), \qquad
         \nu \in\until{\frac{(q_1\cdot \ldots \cdot q_l)-1}{2}}.
      \end{split}
    \end{equation*}
    Also, the orthogonal complement to the controllable subspace, $(X_c)^\perp$,
    (respectively the unobservable subspace, $X_{no}$) is spanned by all the
    corresponding eigenvectors of the form
    \begin{equation*}
      \begin{split}
        V_{\nu} =
        \begin{bmatrix}
          v_{\nu}^T & \!\!0 &\!\!-(\Pi v_{\nu})^T &\!\!-v_{\nu}^T &\!\! 0
          % \ldots\\[1ex]
          % &
          % \ldots &0 &(-1)^{\frac{n}{p_j}-1}(\Pi v_{\nu,j})^T
          % &(-1)^{\frac{n}{p_j}-1}v_{\nu,j}^T &0
          & \!\!\!\ldots &\!\!\!(-1)^{\frac{n}{p}}(\Pi v_{\nu})^T
        \end{bmatrix}^T
      \end{split}
    \end{equation*}
    where $v_{\nu}\in\real^{((q_1\cdot\ldots\cdot q_l)-1)/2}$ is the eigenvector
    of $\Lu_{((q_1\cdot\ldots\cdot q_l)-1)/2}$ corresponding to the eigenvalue
    $\lambda_{\nu}$ for $\nu \in\until{((q_1\cdot\ldots\cdot q_l)-1)/2}$.
%         
    % \item if node $i$ satisfies $(n-i) \eqmod{p_j} (i-1)$ for $l\leq k$
    %   distinct prime factors $p_{j_1}\neq \ldots \neq p_{j_l}$ of $n$, then
    %   the uncontrollable (unobservable) eigenvalues and eigenvectors from $i$ are
    %   all and only the ones computed according to (iii).
        %
    % and the orthogonal complement to the controllable subspace, $(X_c)^\perp$,
    % (the unobservable subspace, $X_{no}$,) is spanned by all the corresponding
    % eigenvectors obtained as in \eqref{eq:unobs_eigvecs_thm}.
  \end{enumerate} \oprocend
\end{theorem}

\begin{remark}[General version of Theorem~\ref{thm:main_thm_path}]
  \label{rmk:general_main_thm_path}
  In the general case of a path graph of length $n = 2^{n_0} \prod_{\nu =1}^k
  p_\nu$, where $p_1, \ldots, p_k$ are not all distinct, statement (i) and (ii)
  of Theorem~\ref{thm:main_thm_path} continue to hold in the same form. As
  regards statement (iii), it still holds in the same form, but it can also be
  strengthen with a slight modification. That is, for each multiple factor
  $\bar{p}$ with multiplicity $\bar{k}$, the statement continues to hold if
  $\bar{p}$ is replaced by $\bar{p}^\alpha$ with
  $\alpha\in\until{\bar{k}}$. 
  % That is, it holds also for the sets $I_s^{p^\alpha} = \{ \ell p^\alpha -
  % \frac{p^\alpha-1}{2}\}_{\ell\in\until{\frac{n}{p^\alpha}}}$.
  Statement (iv) holds if for each prime factor $\bar{p}$ with multiplicity
  $\bar{k}$ we check if node $i$ belongs not only to $I_s^{\bar{p}}$, but also
  to each $I_s^{\bar{p}^\alpha}$ with $\alpha\in\until{\bar{k}}$. Consistently
  the uncontrollable (unobservable) eigenvalues and eigenvectors considered in the
  statement must be constructed by using $\bar{p}^{\bar{\alpha}}$ instead of
  $\bar{p}$, where $\bar{\alpha} = \max_{\alpha} \{\alpha \in \until{\bar{k}} |
  i\in I_s^{\bar{p}^\alpha}\}$. \oprocend
\end{remark}

% The following corollary follows straight from Theorem~\ref{thm:main_thm_path}
% and characterizes all and only the path graphs that are observable from any
% node.

% \begin{corollary}
%   Given a path graph of length $n=2^k$ for some $k\in\natural$, then the path is
%   observable from any node.\oprocend
% \end{corollary}

% \bibliographystyle{IEEEtran}
% \bibliography{alias,ctr_obs,GN}

% \todo{GN: add a remark to compare our result with the one of Godsil also in
%   light of the results of Domenico. ATTENZIONE: matrice di adiacenza}

 \end{document}